\documentclass[a4paper,12pt,reqno,oneside]{amsart} 

\usepackage{setspace} 
\usepackage{typearea} 
\usepackage{amscd,amsmath,amssymb,verbatim} 

\usepackage[dvipsnames]{xcolor}
\usepackage{xr-hyper}
\usepackage{cite}
\usepackage[colorlinks,backref,final,bookmarksnumbered,bookmarks]{hyperref}
\usepackage[backrefs]{amsrefs}
\hypersetup{linkcolor=bluex,citecolor=Green,filecolor=cyan,urlcolor=magenta}
\usepackage{graphicx}
\definecolor{bluex}{Hsb}{220, 1, 1}

\newcommand{\borderref}[2]{{\hypersetup{linkcolor=black}\hyperref[#1]{#1 #2}}}

\usepackage{xfrac}

\usepackage{ifthen}
\newcommand{\arxiv}[2][]{\ifthenelse{\equal{#1}{}}
{\href{http://arxiv.org/abs/#2}{\tt arXiv:#2}}
{\href{http://arxiv.org/abs/math/#2}{\tt arXiv:math.#1/#2}}}

\theoremstyle{plain}
\newtheorem{theorem}{Theorem}[section]
\newtheorem*{theorem*}{Theorem}
\newtheorem{mainthm}{Theorem}

\newtheorem{innercj}{Conjecture}
\newenvironment{cj}[1]
  {\renewcommand\theinnercj{#1}\innercj}
  {\endinnercj}
\newtheorem{lemma}[theorem]{Lemma}
\newtheorem{proposition}[theorem]{Proposition}
\newtheorem{corollary}[theorem]{Corollary}
\newtheorem*{corollary*}{Corollary}

\newtheorem{conjecture}[theorem]{Conjecture}
\newtheorem*{problem*}{Problem}
\newtheorem*{addendum*}{Addendum}

\theoremstyle{definition}
\newtheorem{remark}[theorem]{Remark}
\newtheorem*{remark*}{Remark}
\newtheorem{example}[theorem]{Example}
\newtheorem*{examples*}{Examples}

\newtheoremstyle{named}{}{}{\itshape}{}{\bfseries}{.}{.5em}{\thmnote{#3}}
\theoremstyle{named}
\newtheorem{namedthm}{}

\def\x{\times}
\def\but{\setminus}
\def\emb{\hookrightarrow}
\def\eps{\varepsilon}
\def\phi{\varphi}
\def\emptyset{\varnothing}
\renewcommand{\:}{\colon}

\def\N{\mathbb{N}}
\def\R{\mathbb{R}}
\def\Q{\mathbb{Q}}
\def\Z{\mathbb{Z}}

\def\K{\mathcal{K}}
\def\L{\mathcal{L}}
\def\LM{\mathcal{LM}}

\def\CC{\mathcal{C}}

\def\xr#1{\xrightarrow{#1}}

\makeatletter
\newcommand{\xR}[2][]{\ext@arrow 0359\Rightarrowfill@{#1}{#2}}
\newcommand{\xL}[2][]{\ext@arrow 0359\Leftarrowfill@{#1}{#2}}
\makeatother

\DeclareMathOperator{\lk}{lk}

\def\varnabla{\mbox{\scalebox{1.2}{$\triangledown$}\hspace{-8.2pt}\raisebox{2pt}{\scalebox{0.9}{$\triangledown$}}}}

\begin{document}
\title{Topological isotopy and finite type invariants} 
\author{Sergey A. Melikhov}
\address{Steklov Mathematical Institute of Russian Academy of Sciences, 8 Gubkina st., 119991 Moscow, Russia}
\email{melikhov@mi-ras.ru}

\begin{abstract}
In 1974, D. Rolfsen asked: If two PL links in $S^3$ are isotopic (=homotopic through embeddings), then are they PL isotopic?
We prove that they are PL isotopic to another pair of links which are indistinguishable from each other by finite type invariants.
Thus if finite type invariants separate PL links in $S^3$, then Rolfsen's problem has an affirmative solution.
In fact, we show that finite type invariants separate PL links in $S^3$ if and only if Rolfsen's problem has an affirmative solution
and certain 5 other (rather diverse) conjectures hold simultaneously. 

We also show that if $v$ is a finite type invariant (or more generally a colored finite type invariant) of PL links, and $v$ 
is invariant under PL isotopy, then $v$ assumes the same value on all sufficiently close $C^0$-approximations of 
any given topological link; moreover, the extension of $v$ by continuity to topological links is an invariant of isotopy.
Some specific invariants of this kind are discussed.
\end{abstract}

\maketitle

\section{Introduction} \label{intro}

\subsection{PL isotopy versus topological isotopy} D. Rolfsen posed the following problem in 1974 \cite{Ro1}.
It can be argued that the same problem is also implicit in J. Milnor's 1957 paper, where he studied in detail 
some conditions on isotopies which hold trivially for PL isotopies \cite{Mi2}*{Remark 2 and Theorem 10}.

\begin{namedthm}[Rolfsen's Problem] \label{Rolfsen's} ``If $L_0$ and $L_1$ are PL links%
\footnote{It is not clear whether the links are meant to be in $S^3$ (as is implicitly assumed in the assertion
``all PL knots are isotopic to one another'' mentioned in the previous problem in the same list of problems) 
or in an arbitrary $3$-manifold.
For the sake of definiteness we will assume the former interpretation.}
connected by a topological isotopy, are they PL isotopic?''
\end{namedthm}

It has been long known that some PL knots are topologically slice but not PL slice (see \cite{Kir}*{comments on Problems 1.36--1.38})
and some PL links are connected by a topological $I$-equivalence (=``non-locally-flat concordance''), but not by a PL $I$-equivalence 
\cite{Gi2}, \cite{Ro3}*{Example 7}+\cite{Gi1} (see also \cite{Da}*{\S12}, \cite{Hi}*{Theorem 1.9}).
In fact, there exists a PL link which is topologically slice, but does not bound PL disks disjointly embedded in the $4$-ball (see \cite{Fre}, 
\cite{Kir}*{comments on problems 1.38--1.39} and \cite{Ro3}*{Theorem 3}).

More recently, it was shown by the author that not every topological link in $S^3$ is isotopic to a PL link \cite{M21}.
In the present paper we prove

\begin{mainthm} \label{rolfsen}
If $L_0$ and $L_1$ are PL links in $S^3$ connected by a topological isotopy, then they are PL isotopic to links
$L_0'$ and $L_1'$ which are not separated from each other by finite type invariants.
\end{mainthm}

The definition of finite type invariants is recalled in \S\ref{fti}.

Thus Rolfsen's Problem is solved affirmatively provided that the following conjecture holds.
(By ``links'' we will mean ``PL links'' in the present paper.)

\begin{cj}{(L)}\label{habiro} Finite type invariants separate links in $S^3$.
\end{cj}

This is a natural extension of a well-known conjecture of V. Vassiliev:

\begin{cj}{(K)}[Vassiliev \cite{Va0}, \cite{Va}]\label{vassiliev} Finite type invariants separate knots in $S^3$.
\end{cj}

It should be noted that finite type invariants fail to separate knots in the Whitehead manifold, and more generally in 
any contractible open $3$-manifold which is embeddable in $\R^3$ but not homeomorphic to $\R^3$ \cite{Ei}, \cite{Ei'}.

In the case of knots and links in $S^3$, some weight to Conjectures \ref{vassiliev} and \ref{habiro} is added by the close relation between 
finite type invariants and configuration space integrals (see \cite{Ko}).
Some numeric evidence exists (see e.g.\ \cite{LT}) for the AJ Conjecture, which implies \cite{DuG} that the colored Jones polynomial, 
and hence finite type invariants \cite{MeMo}, detect the unknot.
Some numeric evidence exists (see e.g.\ \cite{Zh}, \cite{Wo}) for the Volume Conjecture, which also implies (see \cite{Tan}) that the colored 
Jones polynomial detects the unknot.

On the other hand, it remains unknown, according to \cite{CDM}*{Preface and \S3.2.4}, whether finite type invariants of knots can ever detect
knot orientation (although a certain type 22 invariant of 6-component links \cite{Lin} and a certain type 7 invariant of 2-component 
string links \cite{DK} do detect link orientation), and whether the signature of a knot is dominated by finite type invariants (see, 
however, \cite{Ga} for some numeric evidence for a conjectural domination).
Closely related to Conjectures \ref{vassiliev} and \ref{habiro} is the following problem: Can one extract at least one
finite type invariant from the celebrated categorified invariants, such as Khovanov homology (which is known to detect the unknot)
and the link Floer homology of Ozsv\'ath--Szab\'o, other than the ones already extractable from what they categorify 
(i.e.\ the Jones polynomial and the Conway potential function, respectively, in the two said cases)?
While there is some work in this direction \cite{Au}, \cite{IY}, \cite{Yo}, it appears that the 
extraction problem turned out to be unexpectedly difficult.

\subsection{Extension of invariants}
Our approach to topological isotopy is based on the notion of {\it $n$-quasi-isotopy}, which was introduced by the author 
(see \cite{MR1}, \cite{MR2}) but can be traced back to a number of older constructions, from the Penrose--Whitehead--Zeeman--Irwin 
trick to Casson handles to the Homma--Bryant proof of the Chernavsky--Miller codimension three PL approximation theorem 
(see \cite{MR1}*{Remarks (i)--(iii) in \S1.5}).
We recall the definition in \S\ref{quasi}.

\begin{theorem} \cite{MR1}*{Theorem 1.3} \label{isotopy0}
Let $h_t\:\Theta\to M$ be a topological isotopy of a closed $1$-manifold in a $3$-manifold.
Then for each $n\in\N$ there exists an $\epsilon>0$ such that every PL homotopy $f_t\:\Theta\to M$ which is $\epsilon$-close to $h_t$ and has only 
finitely many double points occurring at distinct time instances is an $n$-quasi-isotopy.
\end{theorem}

Since the proof is only sketched in \cite{MR1}, we include a more detailed proof in \S\ref{quasi}.

Our next result involves {\it colored finite type invariants} of links, introduced by Kirk and Livingston \cite{KL}
(though they did not call them ``colored'').
An invariant $v$ of links is said to be of {\it colored type $n$} if its standard extension to singular links (see \S\ref{fti})
vanishes on those singular links with $n+1$ double points whose all double points are self-intersections of components (rather 
than intersections of distinct components).%
\footnote{More generally, for a colored link one would consider all double points that are intersections of components of the same color,
but we do not need this more general setting in the present paper.}
Colored type $n$ invariants agree on any two links related by self $C_n$-equivalence (see Remark \ref{self-colored}).
One reason for our interest in colored finite type invariants is the following

\begin{namedthm}[Kirk--Livingston Conjecture] \cite{KL}*{p.\ 1333} \label{kl-conj} There exists an infinite family of linearly 
independent $\Z$-valued colored type $2$ invariants of two-component links with $\lk=0$.
\end{namedthm}

The following is the core result of the present paper.

\begin{mainthm}\label{main1}
Let $v$ be a type $n$ invariant, or more generally a colored type $n$ invariant of links in an oriented $3$-manifold.
If $v$ is invariant under PL isotopy, then it is invariant under $n$-quasi-isotopy.
\end{mainthm}

Theorems \ref{isotopy0} and \ref{main1} imply

\begin{corollary} \label{main1''}
Let $v$ be a finite type invariant, or more generally a colored finite type invariant of links in an oriented $3$-manifold.
Suppose that $v$ is invariant under PL isotopy.

Then $v$ assumes the same value on all sufficiently close $C^0$-approximations of any given topological link. 
Moreover,%
\footnote{This ``Moreover'' can be replaced by ``Consequently'', that is, the second assertion of Corollary \ref{main1''}
can in fact be deduced directly from the first one (using that $[0,1]$ is compact).}
the extension of $v$ by continuity to topological links is an invariant of isotopy.
\end{corollary}

\begin{example} \label{main-example} Let $L$ be a link in $S^3$ with components $K_1,\dots,K_m$.
The hypothesis of Corollary ~\ref{main1''} is satisfied for each coefficient of each of the following 3 formal power series:
\smallskip

\begin{itemize}\setlength\itemsep{5pt}
\item $\bar\nabla_L(z):=\dfrac{\nabla_L(z)}{\nabla_{K_1}(z)\cdots\nabla_{K_m}(z)}$, where $\nabla_L$ is the Conway polynomial 
(see \S\ref{conway} for the details).
\item $\bar\varnabla_L(z_1,\dots,z_m):=\dfrac{\varnabla_L(z_1,\dots,z_m)}{\nabla_{K_1}(z_1)\cdots\nabla_{K_m}(z_m)}$,
where $\varnabla_L$ is the multi-variable Conway polynomial of \cite{M24-2}, which contains the same information as the Conway potential function 
(see \cite{M24-2} for the details).
The coefficients of $\varnabla_L$ and $\bar\varnabla_L$ are not finite type invariants, but they are colored finite type invariants.
\item $\bar Z(L):=\dfrac{Z(L)}{Z(K_1)\cdots Z(K_m)}$, where $Z(L)$ is the Kontsevich integral (see Remark \ref{kontsevich})
and the meaning of the fraction is explained in the proof of Theorem \ref{rational}.
\end{itemize}
\smallskip

Corollary \ref{main1''} implies that $\bar\nabla_L$, $\bar\varnabla_L$ and $\bar Z(L)$ extend by continuity
to invariants of isotopy of topological links.
In the case of $\bar\nabla_L$ and $\bar\varnabla_L$ more can be said: since $\nabla_L$ and $\varnabla_L$ are genuine polynomials,
$\bar\nabla_L$ and $\bar\varnabla_L$ are rational power series for every PL link $L$.
Their extensions to topological links need not be rational (because exactly how close is ``sufficiently close'' $C^0$-approximation 
is determined individually for each coefficient), but whenever one of them is not rational for some topological link $\L$, 
we immediately know that $\L$ is not isotopic to any PL link.
 
The latter observation is applied in a subsequent paper by the author \cite{M24-3} to obtain some progress on another 50-year old
problem of Rolfsen \cite{Ro1}: {\it Is every topological knot isotopic to a PL knot? In particular, is the Bing sling isotopic to a PL knot?}
For example, we show in \cite{M24-3}, using $\bar\nabla_L(z)$ (extended to topological links by means of Corollary \ref{main1''}),
that the Bing sling is not isotopic to any PL knot through topological knots that are intersection of solid tori.
Also we show in \cite{M24-3}, using $\bar\varnabla_L(z_1,z_2)$ (extended to topological links by means of Corollary \ref{main1''}),
that the Bing sling is not isotopic to any PL knot by an isotopy which extends to an isotopy, or just a link homotopy of
two-component links with linking number $1$.

Although the results of \cite{M24-3} go much further than this, Rolfsen's Bing sling problem remains open.
In fact the results of \cite{M24-3} seem to suggest that the Conway potential function is simply not powerful enough to suffice for its solution.
\end{example}

\begin{remark}
It was noted already by Traldi \cite{Tr2}*{Theorem 1 and \S5} and Rolfsen \cite{Ro4} that for any link polynomial $f_L$ which is multiplicative 
under addition of a local knot, the rational function $f_L/(f_{K_1}\cdots f_{K_m})$ is invariant under PL isotopy.
Thus similarly to Example \ref{main-example} one can deal with appropriate power series expansions of the two-variable HOMFLY and Kauffman 
polynomials whose coefficients are finite type invariants (see \cite{BN}*{Theorem ~3}, \cite{Lie}, \cite{HG}).
But it seems that the currently most promising source of invariants which one could try to feed into Corollary \ref{main1''} is 
the Akutsu--Deguchi--Ohtsuki polynomials, which generalize the multivariable Alexander polynomial and in the case of knots 
contain the same information as the colored Jones polynomial \cite{Wil}*{p.\ 24}, \cite{Wil2}.
\end{remark}

\begin{remark} 
Another interesting question which one could try to solve by applying Corollary \ref{main1''} is the Isotopic Realization Problem in $S^3$ 
\cite{MR1}*{Problem 1.1} (see also \cite{M02}*{Questions I and III}):
{\it Is it true that for every continuous map $f\:\Theta\to S^3$, where $\Theta$ is a compact $1$-manifold, there exists a homotopy 
$h_t\:\Theta\to S^3$ such that $h_1=f$ and $h_t$ is injective for $t<1$?}
Potential counterexamples include the infinite connected sum of Whitehead string links \cite{MR1}*{Figure 1} and
a pair of ``linked'' wild arcs \cite{M02}*{Figure 1}.%
\footnote{Beware that the argument in \cite{M02}*{Example 1.3} is erroneous, as explained in \cite{MR1}*{end of \S1.2}.}
As pointed out in \cite{MR1}, to solve the problem negatively it suffices to find an invariant of $1$-quasi-isotopy, or of $k$-quasi-isotopy 
for some fixed $k$, which ``detects accumulation of complexity'' in the sense explained in \cite{MR1}*{Problem 1.5}.
In terms of invariants with values in an abelian group, ``detecting accumulation of complexity'' translates roughly into an infinite series of 
linearly independent invariants (for instance, the infinite family of coefficients of Conway polynomials of knots yields a lower bound on knot genus),
so in view of Theorem \ref{main1}, the desired invariant almost certainly exists if the Kirk--Livingston Conjecture is true.
But there might be other approaches as well, as the gap between $k$-quasi-isotopy and colored type $k$ invariants, well-defined up to PL isotopy,
appears to be very wide (see Examples \ref{weak} and \ref{milnor-double}).
\end{remark}

\begin{example} \label{conway-ex}
Let us discuss in more detail the formal power series $\bar\nabla_L$ of Example \ref{main-example}.
It is well-known (see \S\ref{conway}) that for an $m$-component PL link $L$ its Conway polynomial is of the form 
\[\nabla_L(z)=z^{m-1}(c_0+c_1z^2+c_2z^4+\dots+c_nz^{2n}).\]
Hence the power series $\bar\nabla_L=\nabla_L/(\nabla_{K_1}\cdots\nabla_{K_m})$ is of the form \[\bar\nabla_L(z)=z^{m-1}(\alpha_0+\alpha_1z^2+\alpha_2z^4+\dots).\]
It is not hard to see that $\alpha_k(L)$ is of type $m-1+2k$ and of colored type $2k$. 
With more work one can show that $\alpha_k(L)$ is of colored type $2k-1$ for $k>0$ (Proposition \ref{reduced-coefficients}) 
but $\alpha_2(L)$ is not of colored type $2$ for two-component links of linking number $0$ (Proposition \ref{ctype2}).
On the other hard, we show that $\alpha_k(L)$ is invariant under $k$-quasi-isotopy (Theorem \ref{conway-quasi}). 
Thus the converse to Theorem \ref{main1} is not true.

If $L$ is a $2$-component link, then $\alpha_0(L)=c_0(L)$ is its linking number and 
$\alpha_1(L)=c_1(L)-c_0(L)\big(c_1(K_1)+c_1(K_2)\big)$ is its generalized Sato--Levine invariant, which for any fixed value of $\lk(L)$
generates the group of colored type $1$ invariants modulo colored type $0$ invariants \cite{KL}.
The generalized Sato--Levine invariant emerged independently in the work of Traldi \cite{Tr2}*{\S10}, Polyak--Viro (see \cite{AMR}),
Kirk--Livingston \cite{KL} (see also \cite{Liv}), Akhmetiev (see \cite{AR}) and Nakanishi--Ohyama \cite{NO}.
It is proved in \cite{NO} (see also \cite{M18} for an alternative proof) that $\alpha_0$ and $\alpha_1$ constitute a complete set of 
invariants of self $C_2$-equivalence; and it is shown in \cite{M18}*{Corollary 5.2} that self $C_2$-equivalence is the same thing as 
$\frac12$-quasi-isotopy.

Let us note that the invariance $\alpha_1(L)$ under $1$-quasi-isotopy along with Theorem \ref{isotopy0} already suffice to prove that 
there exists a PL knot in the solid torus $S^1\x D^2$ which is homotopic, but not topologically isotopic to the core circle 
\cite{MR1}*{proof of Theorem 2.2}.
By contrast, every PL knot in the solid torus which is homotopic to the core circle is topologically $I$-equivalent to the core circle, 
i.e.\ they cobound an embedded annulus in $S^1\x D^2\x I$ \cite{Gi1} (see also \cite{Da}*{\S12}, \cite{Hi}*{Theorem 1.9}); this is also 
true of topological knots in the solid torus \cite{Gi2} (see also \cite{M24-3}*{Appendix \ref{part3:app}},
\cite{An0}). 

When $L$ is a $3$-component link, $\alpha_1(L)$ is also of some interest.
Upon adding to $\alpha_1(L)$ a correction term, which is a function of $\alpha_0(\Lambda)$ and $\alpha_1(\Lambda)$ for $2$-component
sublinks $\Lambda$ of $L$, we obtain an invariant $\gamma(L)$ which has a rather neat crossing change formula for self-intersections
of a component (Proposition \ref{gamma}).
\end{example}

\subsection{Reduction of conjectures}
In order to state some further consequences of Theorem \ref{main1} we introduce the following

\begin{cj}{(L/K)} \label{links-modulo-knots} PL isotopy classes of links in $S^3$ are separated by those finite type invariants 
that are well-defined up to PL isotopy.
\end{cj}

Since two links are PL isotopic if and only if they are equivalent under the equivalence relation generated by ambient isotopy 
and insertion of local knots (see \cite{Ro2}*{Theorem 4.2}), this can be seen as the ``links modulo knots'' version of 
the Vassiliev Conjecture (Conjecture \ref{vassiliev}).

Is \ref{habiro} equivalent to the conjunction of \ref{vassiliev} and \ref{links-modulo-knots}?
This is not clear. 
But we prove

\begin{mainthm} \label{string+rational}
\ref{habiro} is equivalent to the conjunction of \ref{vassiliev} and \ref{links-modulo-knots}:

(a) if only rational finite type invariants are considered;

(b) if string links are considered instead of links in $S^3$.
\end{mainthm}

For a more accurate formulation see Theorems \ref{string} and \ref{rational}.
Part (a) is an easy consequence of the (not so easy) theory of the Kontsevich integral; part (b) is proved by using the clasper theory 
of Gusarov and Habiro.

Now we resume the discussion of consequences of Theorem \ref{main1}.

\begin{corollary} If $W$ is any contractible open $3$-manifold other than $\R^3$, then the version of
Conjecture \ref{links-modulo-knots} for links in $W$ fails.
\end{corollary}

\begin{proof} $W$ is known to contain a knot $K$ which is not PL isotopic to the unknot $U$, but $n$-quasi-isotopic to $U$ for all finite $n$ 
(and in fact even for $n=\omega$) \cite{MR2}*{Proposition 2.3}.
Then by Theorem \ref{main1}, $K$ is not separated from $U$ by finite type invariants, well-defined up to PL isotopy.
\end{proof}

\begin{corollary} \label{main1'} Conjecture \ref{links-modulo-knots} is equivalent to the conjunction of the 4 conjectures 
asserting that the following 4 implications can be reversed for any links $L$, $L'$ in $S^3$:
\bigskip

\noindent\fbox{\parbox{\textwidth}{\small\rm
\begin{center}
\smallskip

$L$ and $L'$ are PL isotopic

$\Downarrow$

$L$ and $L'$ are topologically isotopic

$\Downarrow$

$L$ and $L'$ are $n$-quasi-isotopic for all finite $n$

$\Downarrow$

$L$ and $L'$ are not separated by colored finite type invariants, well-defined up to PL isotopy

$\Downarrow$

$L$ and $L'$ are not separated by finite type invariants, well-defined up to PL isotopy
\smallskip

\end{center}
}}
\end{corollary}

\begin{proof} The 4 implications do hold: the first one and the last one are trivial, the second one follows from 
Theorem \ref{isotopy0} and the third one from Theorem \ref{main1}.
Conjecture \ref{links-modulo-knots} is the composite of the converses of the 4 implications.
Hence it is implied by these 4 converses, and at the same time implies that all 5 assertions in the frame are equivalent.
\end{proof}

\begin{corollary} \label{old}
If Conjecture \ref{links-modulo-knots} holds, then topologically isotopic links in $S^3$ are PL isotopic.
\end{corollary}

While Corollary \ref{old} and Theorem \ref{string+rational} do not quite suffice to get Theorem \ref{rolfsen}, 
a minor modification of their proofs does.
On this way we also get a modification of Corollary \ref{main1'}, which we now prepare to state.

If two string links are not separated by type $n$ invariants, then it is easy to see that their closures are also 
not separated by type $n$ invariants (cf.\ Lemma \ref{fti-closure}).
The converse is false, and for a good reason: a nontrivial string link may have trivial closure.%
\footnote{For instance, the connected sum $W\#\rho W$ of the Whitehead string link and its reflection is a non-trivial 
string link, as detected by a type $3$ invariant (see \cite{MR1}*{proof of Theorem 2.2} or \cite{M18}*{Example 3.3}).
On the other hand, the closure of $W\#\rho W$ (which is a connected sum of the Whitehead link with its reflection along bands, 
orthogonal to the mirror) is easily seen to be the unlink.}
But if we fix the link and allow the string link to vary, the converse is still false.
For instance, the Borromean rings are not separated from the unlink by type $2$ invariants \cite{Hab}*{Proposition 7.4(2)}, 
but whenever we represent both as closures of string links, these string links are always separated by type $2$ invariants, 
since it is known that string links that are not separated by type $2$ invariants are $C_3$-equivalent (see \cite{MY}), 
but the Borromean rings are not $C_3$-equivalent to the unlink \cite{Hab}*{Proposition 7.4(1)}.

Moreover, given a sequence $L_1,\dots,L_k$ of links such that $L_1$ is the Borromean rings and $L_k$ is the unlink, 
and a representation for each pair $L_i$, $L_{i+1}$ as closures of string links $\Lambda_{2i-1}$, $\Lambda_{2i}$,
then at least one pair $\Lambda_{2i-1}$, $\Lambda_{2i}$ will be separated by type $2$ invariants, as the same argument with 
$C_3$-equivalence shows.
In this event we say that the Borromean rings and the unlink ``are separated by type $2$ invariants of string links''.
More formally, given a class $C$ of invariants of string links, and given links $L$ and $L'$ in $S^3$, we say that $L$ and $L'$ 
are {\it not separated by invariants of string links of class $C$} if $L$ and $L'$ are equivalent with respect to the equivalence 
relation generated by ambient isotopy and the relation ``to be closures of string links that are not separated by invariants 
of class $C$''.
If two links in $S^3$ are not separated by type $n$ invariants of string links, then by the previous arguments 
(cf.\ Lemma \ref{fti-closure}) they are not separated by type $n$ invariants (of links in $S^3$).

\begin{cj}{(HM)} \label{habegger-meilhan} If two links in $S^3$ are not separated by finite type invariants, then 
for each $n$ they are not separated by type $n$ invariants of string links.
\end{cj}

The question of validity of Conjecture \ref{habegger-meilhan} is presumably included in a problem of N.~ Habegger and J.-B. Meilhan 
\cite{HM}*{Problem 5.4}, which however can be made precise in multiple ways as it is stated in informal language.

\begin{mainthm} \label{main2}
\ref{habiro} is equivalent to the conjunction of \ref{vassiliev}, \ref{habegger-meilhan} and the 4 conjectures asserting that 
the following 4 implications can be reversed for any links $L$, $L'$ in $S^3$:
\bigskip

\noindent\fbox{\parbox{\textwidth}{\small\rm
\begin{center}
\smallskip

$L$ and $L'$ are PL isotopic

$\Downarrow$

$L$ and $L'$ are topologically isotopic

$\Downarrow$

$L$ and $L'$ are $n$-quasi-isotopic for all finite $n$

$\Downarrow$

for each $n$, $L$ and $L'$ are not separated by colored type $n$ invariants of string links, well-defined up to PL isotopy

$\Downarrow$

for each $n$, $L$ and $L'$ are not separated by type $n$ invariants of string links, well-defined up to PL isotopy
\smallskip

\end{center}
}}
\end{mainthm}

\subsection{Organization of material}
The paper is organized as follows.
\begin{itemize}
\item A detailed proof of Theorem \ref{isotopy0} is included in \S\ref{quasi}.
\item Theorem \ref{main1} is proved in \S\ref{fti}.
\item The main result of \S\ref{cn-equivalence} says roughly that finite type invariants of string links are dominated 
by those ones that are additive under insertion of local knots (Theorem \ref{locally-additive2}). 
The proof is based on the clasper calculus.
\item Theorem \ref{rolfsen} is proved in \S\ref{proof-rolfsen}, based on the results of \S\S\ref{quasi}--\ref{cn-equivalence}.
\item Theorem \ref{main2} is proved in \S\ref{proofs2}, based on the results of \S\S\ref{quasi}--\ref{proof-rolfsen}. 
\item Theorem \ref{string+rational} is proved in \S\ref{proofs3}, based on the results of \S\ref{cn-equivalence} and 
\cite{M24-1'}.
\item \S\ref{conway} is devoted to examples, illustrating colored finite type invariants and $n$-quasi-isotopy 
in a more practical light in the context of the Conway polynomial.
\end{itemize}

As a byproduct of the proofs of the main results, the paper also contains alternative proofs of some known results:
\begin{itemize}
\item \S\ref{cn-equivalence} includes a simple visual proof that type $n$ invariants are invariant
under $C_n$-equivalence (Proposition \ref{cn-eq}).
While the proofs of this fact in the literature are not very difficult, the new proof appears to be still easier.
\item \S\ref{cn-equivalence} also contains an alternative proof of a 2009 theorem by G. Massuyeau, which provides 
a partial converse to the previous item (Theorem \ref{massuyeau}).
Our argument proceeds in the original language of string links (as opposed to Massuyeau's language of homology cylinders)
and is a correction of K. Habiro's attempted proof of a slightly weaker result.
\item \S\ref{fti} contains a particularly simple (but asymmetric) form of the Leibniz rule for finite type invariants,
which I did not see in the literature.
\end{itemize}

\section{Basic notions} \label{basic}

By a {\it link} we mean a PL embedding of a closed $1$-manifold in $3$-manifold.
Thus every link is of the form $mS^1\to M$, where $mS^1=\{1,\dots,m\}\x S^1$.
More generally, a {\it tangle} is a proper PL embedding $L\:\Theta\to M$ of compact $1$-manifold in a $3$-manifold, 
where {\it proper} means that $L^{-1}(\partial M)=\partial\Theta$.
Two tangles are called {\it equivalent} if they are ambient isotopic keeping $\partial M$ fixed.
An embedding $\Xi\:\partial\Theta\to\partial M$ is called a {\it boundary pattern}, and a tangle or more generally
a map $L\:\Theta\to M$ is said to be {\it of boundary pattern} $\Xi$ if $L|_{\partial\Theta}=\Xi$.
A {\it string link} is a tangle $mI\to I^3$, where $I=[0,1]$ and $mI=\{1,\dots,m\}\x I$, of the 
{\it string link boundary pattern} $\Xi_m\:\partial(mI)\to\partial I^3$, given by $\Xi_m(k,i)=(\frac{k}{m+1},\frac12,i)$.
The {\it string unlink} $U\:mI\to I^3$ is given by $U(k,t)=(\frac{k}{m+1},\frac12,t)$.
The {\it closure} of a string link $L\:mI\to I^3$ is the link 
$mS^1\cong mI\cup_\partial mI\xr{L\cup U}I^3\cup_\partial I^3\cong S^3$.
The {\it unlink} is the closure of the string unlink.
A {\it knot} is a $1$-component link, and the {\it unknot} is the $1$-component unlink.

\section{$n$-Quasi-isotopy and strong $n$-quasi-isotopy}\label{quasi}

We recall the definitions of $n$-quasi-isotopy and strong $n$-quasi-isotopy \cite{MR1}, \cite{MR2}.
Let $f\:\Theta\to M$ be a PL map of a $1$-manifold in a $3$-manifold with precisely one double point $x=f(p)=f(q)$.
It is called a {\it strong $n$-quasi-embedding} if in addition to the singleton $B_0:=\{x\}$
there exist PL $3$-balls $B_1,\dots,B_n\subset M$ and arcs $J_0,\dots,J_n\subset\Theta$ such that 
$f^{-1}(B_i)\subset J_i$ for each $i\le n$ and $B_i\cup f(J_i)\subset B_{i+1}$ for each $i<n$.
Let us note that since $J_n$ is an arc, $B_n$ can intersect the image of only one component of $\Theta$ (namely, the one 
which contains $J_n$).
Next, $f$ is called an {\it $n$-quasi-embedding} if in addition to the singleton $P_0:=\{x\}$ there exist compact
polyhedra $P_1,\dots,P_n\subset M$ and arcs $J_0,\dots,J_n\subset\Theta$ such that $f^{-1}(P_i)\subset J_i$ 
for each $i\le n$ and $P_i\cup f(J_i)\subset P_{i+1}$ for each $i<n$, and moreover the inclusion $P_i\cup f(J_i)\to P_{i+1}$
is null-homotopic for each $i<n$.
It should be noted that like before, $P_n$ can intersect the image of only one component of $\Theta$.
A PL homotopy $f_t\:\Theta\to M$ is called a {\it (strong) $n$-quasi-isotopy} if it contains only finitely many double points, 
all occurring at distinct time instances $t_1,\dots,t_k\in I$, and each $f_{t_i}$ is a (strong) $n$-quasi-embedding.
Two tangles $L_0,L_1\:\Theta\to M$ are {\it (strongly) $n$-quasi-isotopic} if they are of the same boundary pattern $\Xi$ and are 
connected by a (strong) quasi-isotopy $L_t\:\Theta\to M$ such that each $L_t$ is a proper map of boundary pattern $\Xi$.

\begin{figure}[h]
\includegraphics[width=0.6\linewidth]{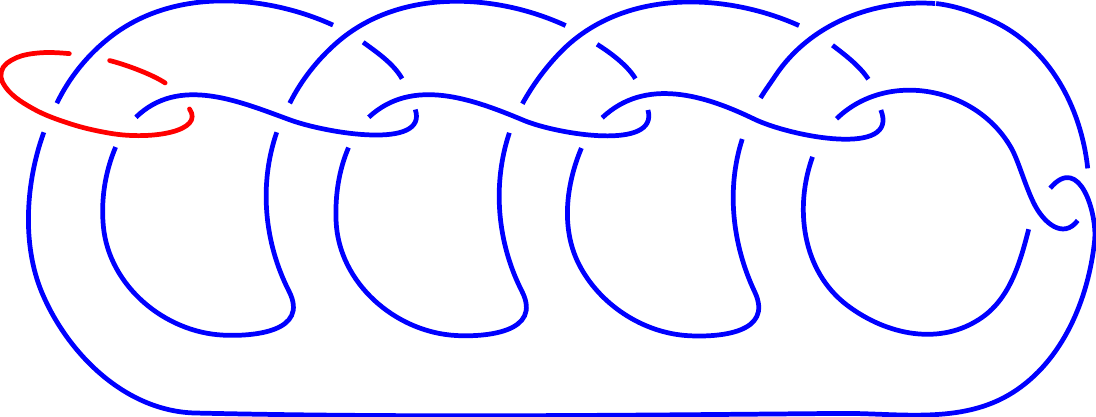}
\caption{The $n$th Milnor link for $n=4$.}
\label{milnor}
\end{figure}

\begin{example} The $n$th Milnor link $M_n$ (see Figure \ref{milnor}) is easily seen to be strongly $(n-1)$-quasi-isotopic to the unlink
(see Figure \ref{milnor-J3}).
It is not $n$-quasi-isotopic to the unlink, which can be shown either ``geometrically'', by using $(n+1)$-cobordism and 
Cochran's derived invariant $\beta^n$ \cite{MR2}*{Corollary 3.6} or ``algebraically'', by using the fundamental group and Milnor's 
invariant $\bar\mu(\underbrace{11\dots11}_{2n}22)$ \cite{MM}*{Theorem 2.12}.
\end{example}

\begin{figure}[h]
\includegraphics[width=\linewidth]{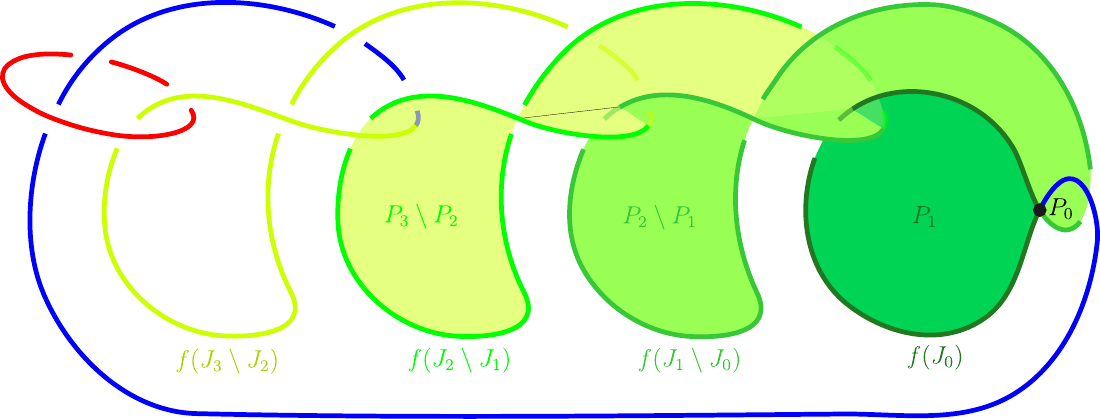}
\caption{The only singular link in a generic homotopy from $M_n$ to the unlink is an $(n-1)$-quasi-embedding $(n=4)$.
Moreover, it is a strong $(n-1)$-quasi-embedding, using that all the $P_i$ are collapsible, and hence their regular neighborhoods are balls.}
\label{milnor-J3}
\end{figure}

\begin{figure}[h]
\includegraphics[width=0.7\linewidth]{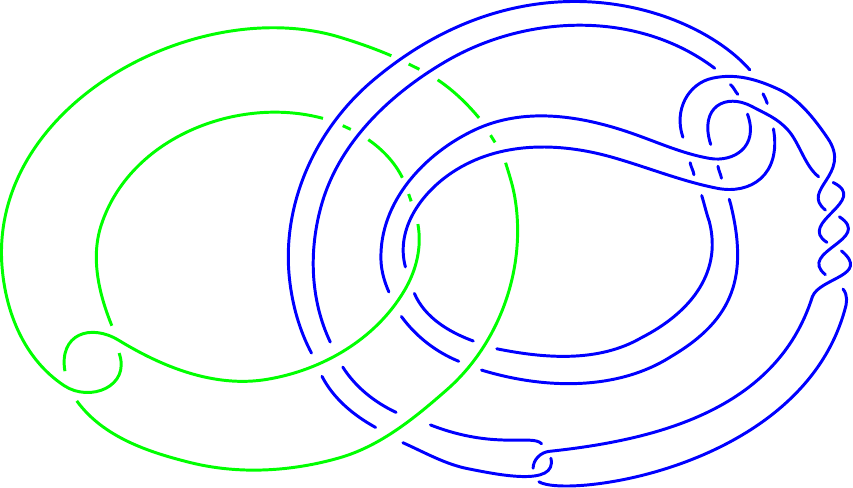}
\caption{The $n$th Whitehead link for $n=3$.}
\label{whitehead}
\end{figure}

\begin{example} \label{wh-n}
The $n$th Whitehead link $W_n$ (see Figure \ref{whitehead}) is the $n$th (untwisted left-handed) Whitehead double 
of either component of the Hopf link (due to the symmetry of the Whitehead link, it does not matter which component is being doubled).
There is an obvious $(n-1)$-quasi-isotopy from $W_n$ to the unlink (see Figure \ref{whitehead-J2}).
It is not hard to show that this particular $(n-1)$-quasi-isotopy is not an $n$-quasi-isotopy and not 
a strong $1$-quasi-isotopy (see \cite{MR1}*{\S1.2}).
\end{example}

\begin{figure}[h]
\includegraphics[width=0.85\linewidth]{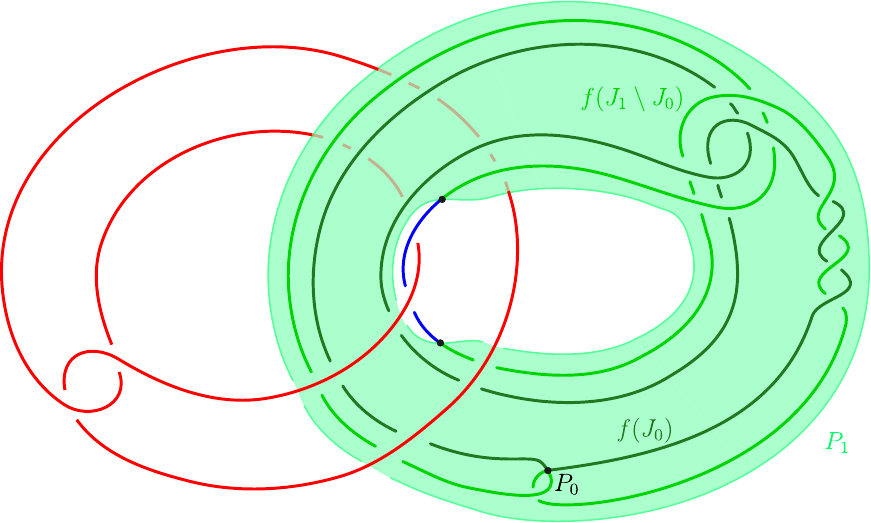}
\bigskip

\includegraphics[width=0.85\linewidth]{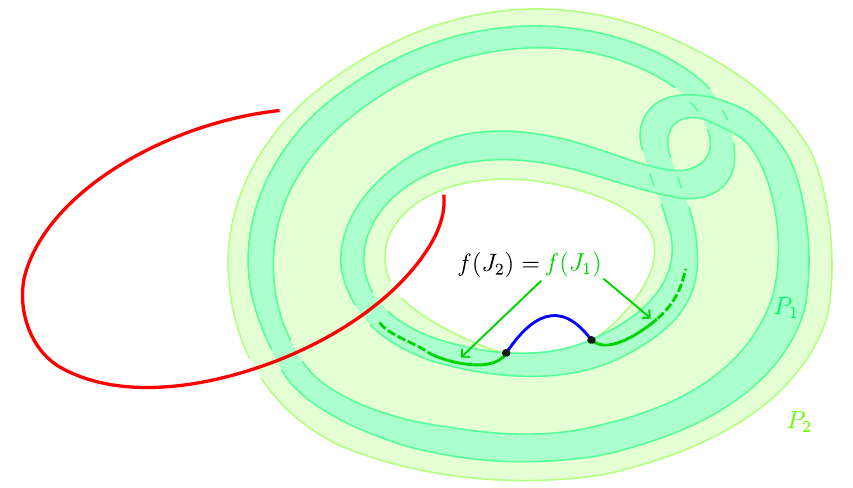}
\caption{The only singular link in a generic homotopy from $W_n$ to the unlink is an $(n-1)$-quasi-embedding $(n=3)$.}
\label{whitehead-J2}
\end{figure}

\begin{conjecture} \cite{MR2} \label{wh-conj}
(a) $W_n$ is not $n$-quasi-isotopic to the unlink.

(b) $W_n$ is not strongly $1$-quasi-isotopic to the unlink.
\end{conjecture}

Let us note that either (a) or (b) would imply (see Theorem \ref{isotopy}) that each $W_n$ is not topologically isotopic to the unlink.
It is known that $W_n$ is topologically slice, and hence topologically I-equivalent to the unlink for $n\ge 3$ \cite{Fre}.

\begin{remark}
Let us discuss some approaches to Conjecture \ref{wh-conj}.
The Conway polynomial vanishes on $W_n$ for $n\ge 2$ (this follows from Lemma \ref{annulus0} and the Conway skein relation (\ref{conway-skein})).
For similar reasons the multi-variable Conway polynomial also vanishes on $W_n$ for $n\ge 2$.
To compute the Jones polynomial of the Whitehead double a link $L$ along some component $K$ is the same task as to compute the Jones polynomial 
of the link obtained from $L$ by adding a parallel copy of $K$ (this follows from the Jones skein relation and the formula \cite{Lic}*{p.\ 26} 
for reversing the orientation of a component).
So this is closely related to the task of computation of the colored Jones polynomial.
A formula for the colored Jones polynomial of a Whitehead double of a knot is given in \cite{Tan} and \cite{Zh}*{\S2}, and 
the colored Jones polynomial of a twisted version of $W_n$ is calculated in \cite{Wo}*{\S4}.
See also \cite{Fu0} for a related computation.
\end{remark}

\begin{remark} As noted in \cite{MR2}*{Theorem 1.2}, the definition of $n$-quasi-isotopy can be considerably simplified for string links:
A PL map $f\:mI\to I^3$ with precisely one double point is an $n$-quasi-embedding if and only if its singular component is 
$(n-1)$-contractible in the complement to the remaining components. 
Here a compact subset $X$ of a manifold $M$ is called {\it $k$-contractible} in $M$ if there exist compact subpolyhedra
$A\subset P_0\subset\dots\subset P_{k+1}\subset M$ such that each inclusion $P_i\to P_{i+1}$ is null-homotopic.
A relation between $n$-quasi-isotopy of links in $S^3$ and $n$-quasi-isotopy of string links is discussed in Proposition \ref{string-quasi}.
\end{remark}

Many further examples and results on $n$-quasi-isotopy can be found in \cite{MR1} and \cite{MR2}.
See also Theorem \ref{conway-quasi} below.

Theorem \ref{isotopy0} is a consequence of the following

\begin{theorem} \cite{MR1}*{Theorem 1.3} \label{isotopy}
Let $h_t\:\Theta\to M$ be a topological isotopy of a compact $1$-manifold in a $3$-manifold.
Then for each $n$ there exists an $\epsilon>0$ such that every PL homotopy $f_t\:\Theta\to M$
which is $\epsilon$-close to $h_t$ and has only finitely many double points occurring at distinct time instances
is a strong $n$-quasi-isotopy.
\end{theorem}

Since the proof sketch given in \cite{MR1} covers explicitly only the case $n=1$, we include the details here.
(The case $n=1$ might be the most important one for the purposes of \cite{MR1}*{Problem 1.1}, but it certainly does not suffice 
for the purposes of the present paper.)

\begin{proof}
To simplify the argument we use a metric on $M$ whose closed $\epsilon$-balls are PL balls.
For instance, for $M=\R^3$ the usual $l_\infty$ metric will do.
In general, one such metric on $M$ is given by the following construction (see \cite{M-up} for its further discussion).
Let $K$ be a triangulation of $M$ and $V$ the set of its vertices.
The set $\R[V]$ of all finite formal linear combinations $\lambda_1v_1+\dots+\lambda_kv_k$, where $v_i\in V$ and $\lambda_i\in\R$, 
is a vector space over $\R$, which can be endowed with the $l_\infty$ metric.
Let $f\:M\to\R[V]$ be defined by sending the barycenter of every simplex $\sigma$ of $K$ to the sum of the vertices of $\sigma$
and extending linearly to every simplex of the barycentric subdivision of $K$.
It is easy to see that $f$ is an embedding and that the $f$-image of each $k$-simplex $\sigma$ of $K$ is a union of $k+1$ of $k$-dimensional
faces of the cube $[0,1]^{k+1}\subset\R^{k+1}=\R[\text{vertices of $\sigma$}]$.
It follows that the induced metric on $M$ is as desired.

\begin{lemma}\label{strong lemma}
For every $\delta>0$ there exists a $\gamma=\gamma(\delta)>0$ such that for any triple $(p,q,t)\in\Theta\x\Theta\x I$ 
such that $h_t(p)$ and $h_t(q)$ are $\gamma$-close, $p$ and $q$ belong to the same component $C$ of $\Theta$ and are not antipodal in it 
(in the event that $C\cong S^1$), and the $h_t$-image of the shortest arc $[p,q]\subset C$ between $p$ and $q$ is of diameter $<\delta$.
\end{lemma}

\begin{proof}
Given a $\delta>0$, let $V_\delta\subset\Theta\x\Theta\x I$ denote the set of all triples $(p,q,t)$ such that $p$ and $q$ 
belong to the same component of $\Theta$ and are not antipodal in it, and $h_t([p,q])$ is of diameter $<\delta$.
It is easy to see that $V_\delta$ is open and contains $\Delta_\Theta\x I$, where $\Delta_X=\{(x,x)\in X\x X\}$.
Then $\Theta\x\Theta\x I\but V_\delta$ is compact, and hence so is its image $K_\delta$ under the map $H\:\Theta\x\Theta\x I\to M\x M$, 
defined by $H(x,y,t)=\big(h_t(x),h_t(y)\big)$.
Since $h_t$ is an isotopy, $K_\delta$ is disjoint from $\Delta_M$.
Then there exists a $\gamma>0$ such that $M\x M\but K_\delta$ contains every pair $(x,y)\in M\x M$ such that 
$x$ and $y$ are at a distance $\le\gamma$.
\end{proof}

Let $\epsilon_{n+1}$ be any positive real number.
Assuming that $\epsilon_i$ is defined, let $\epsilon_{i-1}=\frac12\gamma(\frac{\epsilon_i}2)$ be given by Lemma \ref{strong lemma}.
Let us note that each $\epsilon_{i-1}\le\frac{\epsilon_i}4$.
Finally, set $\epsilon=\epsilon_0$.

Let $f_t\:(\Theta,\partial\Theta)\to (M,\partial M)$ be a PL homotopy, $\epsilon$-close to $h_t$ and possessing only finitely many double points 
that occur at distinct time instances $t_1,\dots,t_k\in I$.
Let us fix some $j\in\{1,\dots,k\}$, let us write $f=f_{t_j}$ and $h=h_{t_j}$, and let $x:=f(p)=f(q)$, $p\ne q$, be 
the corresponding double point.
Let $B_0=\{x\}$ and for $i=1,\dots,n$ let $B_i$ be the closed ball of radius $\epsilon_i$ about $x$.
Thus $B_0\subset\dots\subset B_n$.

Since $h$ is $\eps$-close to $f$, the distance between $h(p)$ and $h(q)$ is at most $2\epsilon=\gamma(\frac{\epsilon_1}2)$.
Hence by Lemma \ref{strong lemma} $p$ and $q$ belong to the same component of $\Theta$ and are not antipodal in it, 
and $h([p,q])$ is of diameter $<\frac{\epsilon_1}2$.
Set $J_0=[p,q]$.
Then $f(J_0)$ is of diameter $<\frac{\epsilon_1}2+2\epsilon\le\epsilon_1$.
Hence it lies in $B_1$.

Let $i\in\{1,\dots,n\}$.
Given any $r\in\Theta$ such that $f(r)\in B_i$, the distance from $f(r)$ to $x=f(p)$ is at most $\epsilon_i$.
Hence the distance from $h(r)$ to $h(p)$ is at most $\epsilon_i+2\epsilon\le 2\epsilon_i=\gamma(\frac{\epsilon_{i+1}}2)$.
Then by Lemma \ref{strong lemma} $r$ belongs to the same component of $\Theta$ as $p$ and is not antipodal to $p$ in it, and $h([p,r])$ 
is of diameter $<\frac{\epsilon_{i+1}}2$.
Hence $f([p,r])$ is of diameter $<\frac{\epsilon_{i+1}}2+2\epsilon\le\epsilon_{i+1}$.
Hence it lies in $B_{i+1}$, unless $i=n$ (in which case $B_{i+1}$ is undefined, because it is not needed).
Let $J_i$ be the union of all arcs of the form $[p,r]$, where $r\in\Theta$ is such that $f(r)\in B_i$.
Clearly $J_i$ is closed, connected and does not contain the antipode of $p$, so it is a closed arc.
By construction $f^{-1}(B_i)\subset J_i$ and if $i<n$ also $f(J_i)\subset B_{i+1}$.
\end{proof}

Suppose that the $1$-manifold $\Theta$ consists of $m$ connected components $\Theta_1,\dots,\Theta_m$.
A {\it (strong) $(n_1,\dots,n_m)$-quasi-isotopy} $\Theta\to M$ is a consecutive composition of homotopies 
such that each of them is a (strong) $n_i$-quasi-isotopy with double points occurring only in $\Theta_i$.

The proof of Theorem \ref{isotopy} works to prove the following

\begin{theorem} \label{isotopy'}
Let $h_t\:\Theta\to M$ a topological isotopy.
Then there exist numbers $\epsilon_0,\epsilon_1,\dots$ such that every PL homotopy $f_t\:\Theta\to M$
whose restriction to every $\Theta_i$ is $\epsilon_{n_i}$-close to $h_t|_{\Theta_i}$ and has only finitely 
many double points occurring at distinct time instances is a strong $(n_1,\dots,n_m)$-quasi-isotopy.
\end{theorem}

\section{Finite type invariants and colored finite type invariants}\label{fti}

By a {\it singular tangle} we mean a proper PL map $L\:\Theta\to M$ of a compact $1$-manifold in a $3$-manifold, 
which has no triple points and only finitely many double points, and every its double point $L(p)=L(q)$ is 
{\it rigid} in the sense that the tangent vectors $dL_p(1)$ and $dL_q(1)$ exist (i.e.\ $L$ has no kinks at $p$ 
and $q$) and hence are linearly independent.
Two singular tangles $L,L'\:\Theta\to M$ are called {\it equivalent} if they are ambient isotopic keeping $\partial M$ 
fixed and so that all the double points remain rigid at every time instant.

For the rest of this section we fix a compact $1$-manifold $\Theta$, an oriented $3$-manifold $M$ and 
a boundary pattern $\Xi\:\partial\Theta\to\partial M$, and by ``tangles'' or ``singular tangles'' mean only 
those of this boundary pattern.
(For example, these can be links in $S^3$ or string links.)

Given an invariant $v$ of tangles, which assumes values in an abelian group $G$, we claim that it uniquely extends 
to a $G$-valued invariant $v^\x$ of singular tangles satisfying the relation
\[v^\x(L)=v^\x(L_+)-v^\x(L_-), \tag{V}\label{vassiliev-extension}\]
where $L$, $L_+$ and $L_-$ agree outside a small ball, and inside this ball they are as follows:%
\footnote{Let us note that $L_+$ and $L_-$ can be distinguished from each other without using the plane diagram.
Namely, the frame $\big(d(L_\epsilon)_p(1),\,d(L_\epsilon)_q(1),\,L_\epsilon(p)-L_\epsilon(q)\big)$ in $T_{L(p)}M$ 
preserves its sign when $p$ and $q$ are interchanged (and also when the orientation of $\Theta$ is reversed), but reverses the sing
when $\epsilon$ is reversed.}
\smallskip

\begin{center}
\includegraphics[width=0.4\linewidth]{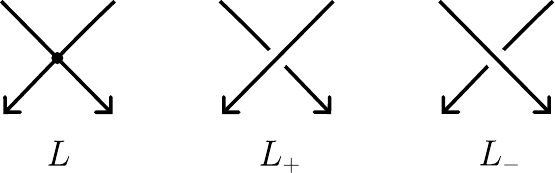}
\end{center}
\smallskip

\noindent
Indeed, the value of $v^\x$ on a given singular tangle can be computed by inductively resolving 
its double points in some order, and applying the relation (\ref{vassiliev-extension}) at each stage.
This results in the formula 
\[v^\x(L)=\sum_{\epsilon_1,\dots,\epsilon_n\in\{1,-1\}}\epsilon_1\cdots\epsilon_n v(L_{\epsilon_1\dots\epsilon_n}) \tag{V$'$}\label{vassiliev-extension'}\]
for every singular tangle $L$ with $n$ double points, where $L_{\epsilon_1\dots\epsilon_n}$ is the tangle obtained by resolving $L$ 
at all its double points, with the resolution at the $i$th double point being of the sign $\epsilon_i$.
But it is easy to see that this formula is not affected by reordering the double points of $L$.
(And this is where we need the group $G$ to be abelian.)

A $G$-valued invariant $v$ of tangles is said to be of {\it type $n$} if $v^\x$ vanishes on 
all singular tangles with $n+1$ double points (and therefore also on all singular tangles with $>n$ double points).
It is said to be of {\it finite type} if it is of type $n$ for some $n$.
Finite type invariants were introduced by M. Gusarov (cf.\ the editor's comment in \cite{Gu2}) and
independently by V. Vasiliev \cite{Va0}; the case of links appears in \cite{Gu0}.

The following lemmas follow easily from the definition.

\begin{lemma} \label{fti-closure} Let $v$ be a type $n$ invariant of $m$-component links in $S^3$.
Let $v_c$ be an invariant of $m$-component string links defined by evaluating $v$ on the the closure.
Then $v_c$ is a type $n$ invariant.
\end{lemma}

\begin{lemma} \label{knot-link} (a) Let $v$ be a type $n$ invariant of knots in $S^3$.
Let $v_i^\lambda$ be an invariant of $m$-component links in $S^3$ defined by evaluating $v$ on the $i$th component.
Then $v_i^\lambda$ is a type $n$ invariant.

(b) Let $v$ be a type $n$ invariant of $m$-component links in $S^3$.
Let $v_i^\kappa$ be an invariant of knots in $S^3$ defined by evaluating $v$ on the totally split link whose $i$th component
is the given knot and the other components are unknotted.
Then $v_i^\kappa$ is a type $n$ invariant.
\end{lemma}

It is easy to see that $G$-valued type $n$ invariants of tangles $\Theta\to M$ of boundary pattern $\Xi$ form an abelian group.
In particular, $G$-valued type $n$ invariants of $m$-component links in $S^3$ form an abelian group $G_{m,n}$.
Lemma \ref{knot-link} identifies the direct product of $m$ copies of $G_{1,n}$ with a direct summand of $G_{m,n}$.

An alternative notion of finite type invariants was introduced by Kirk and Livingston \cite{KL}.
Let us recall that a {\it link map} is a map which sends distinct connected components to disjoint sets.
A $G$-valued invariant $v$ of tangles is said to be of {\it colored type $n$} if $v^\x$ 
vanishes on those singular tangles with $n+1$ double points that are link maps (and therefore also on those singular 
tangles with $>n$ double points that are link maps).
The invariant $v$ is said to be of {\it colored finite type} if it is of colored type $n$ for some $n$.

\begin{example} The colored type $0$ invariant $v(L)=(-1)^{\lk(L)}$ is not of finite type.

Indeed, arguing by induction, it is easy to see that $v^\x(L_s)=\pm 2^k$ for any singular link $L_s$ with
$k$ intersections between distinct components and no self-intersections. 
\end{example}

\begin{example} The coefficients of the Conway polynomial are of finite type, but their colored type is usually less than their type 
(see Lemma \ref{conway-coefficients}).
The coefficients of the multi-variable Conway polynomial are ``naturally occurring'' invariants of colored finite type that are not 
of finite type \cite{M24-2}. 
\end{example}

It is sometimes more natural to consider a different filtration on colored finite type invariants, which is implicit in
\cite{ShYa}*{pp.\ 883, 886}.
Assuming that $\Theta$ consists of $m$ connected components $\Theta_1,\dots,\Theta_m$, the invariant $v$ is said to be 
of {\it type $(k_1,\dots,k_m)$} if $v^\x$ vanishes on all singular links with $k_1+1$ self-intersections of $\Theta_1$,
on all singular links with $k_2+1$ self-intersections of $\Theta_2$, and so on.
It is easy to see that every type $(k_1,\dots,k_m)$ invariant is a colored type $k_1+\dots+k_m$ invariant; 
and every colored type $n$ invariant is a type $(n,\dots,n)$ invariant.
Thus $v$ is of colored finite type if and only if it is of type $(k_1,\dots,k_m)$ for some $k_1,\dots,k_m$.

\begin{lemma} \label{leibniz}
Let $R$ be a ring and $u$, $v$ be $R$-valued invariants of tangles $\Theta\to M$ with a fixed boundary pattern.
Then
\[(uv)^\x(L)=\sum_{\substack{S\cup T=\Delta_L\\ S\cap T=\emptyset}} u^\x(L_{S+})v^\x(L_{T-}),\] 
where for a singular tangle $L$ with a finite set $\Delta_L$ of double points and for an $S\subset\Delta_L$, by $L_{S+}$ 
(resp.\ $L_{S-}$) we denote the partial resolution of $L$ where the double points in $S$ are resolved positively (resp.\ negatively)
and the double points in $\Delta_L\but S$ remain unresolved.
\end{lemma}

When $L$ has only one double point, the Leibniz rule of Lemma \ref{leibniz} specializes to
\[(uv)^\x(L)=u(L_+)v^\x(L)+u^\x(L)v(L_-).\]

A more symmetric, but more complicated Leibniz rule is proved in \cite{Wi}*{Theorem 3.2}.
We will use the specific form of our Leibniz rule in the proof of Proposition \ref{ctype2}.

\begin{proof} Let us note that the lemma holds trivially for non-singular tangles.
Let us assume that it holds for singular tangles with fewer double points than $L$.
Pick some $x\in\Delta_L$.
Then by (\ref{vassiliev-extension}) $(uv)^\x(L)=(uv)^\x(L_{\{x\}+})-(uv)^\x(L_{\{x\}-})$.
Applying the induction hypothesis and abbreviating $(L_{S\epsilon})_{T\delta}$ by $L_{S\epsilon,\,T\delta}$, the latter expression
takes the form
\[\sum_{\substack{S\cup T=\Delta_L\but\{x\}\\ S\cap T=\emptyset}} u^\x(L_{\{x\}+,\,S+})v^\x(L_{\{x\}+,\,T-})-u^\x(L_{\{x\}-,\,S+})v^\x(L_{\{x\}-,\,T-}).\]
If we subtract $u^\x(L_{\{x\}+,\,S+})v^\x(L_{\{x\}-,\,T-})$ from the left hand term and add it to the right hand term, we obtain
\[\sum_{\substack{S\cup T=\Delta_L\but\{x\}\\ S\cap T=\emptyset}} u^\x(L_{\{x\}+,\,S+})v^\x(L_{T-})+u^\x(L_{S+})v^\x(L_{\{x\}-,\,T-}).\]
But this clearly equals the right hand side of (\ref{leibniz}).
\end{proof}

\begin{corollary} \cite{Gu1}, \cite{Wi} \label{product}
(a) If $u$, $v$ are of (colored) types $k$, $l$ respectively, then $uv$ is of (colored) type $k+l$.

(b) If $u$, $v$ are of types $(k_1,\dots,k_m)$ and $(l_1,\dots,l_m)$ respectively, then $uv$ is of type $(k_1+l_1,\dots,k_m+l_m)$.
\end{corollary}

The non-colored case is proved in \cite{Gu1}*{Lemma 5.8}, \cite{Wi}*{Corollary 3.3} and \cite{CDM}*{3.2.3} (see also \cite{BN}*{Exercise 3.10}), 
and the proofs in \cite{Wi}, \cite{CDM} work in both colored cases. 

\begin{remark}
Given a {\it coloring} of $\Theta$, that is, a continuous map $c\:\Theta\to\{1,\dots,\chi\}$, one may define 
a $c$-link map as a map which sends distinct point-inverses of $c$ to disjoint sets.
(One may think of $c$ as dipping the components of $\Theta$ into paintpots numbered $1,\dots,\chi$.)
This leads to invariants of $c$-type $n$.
The case of a constant coloring corresponds to type $n$ invariants, and the case of a coloring which is itself 
a link map corresponds to colored type $n$ invariants.
These are the only two cases that are needed in the present paper.
\end{remark}

Theorem \ref{main1} is a consequence of the following

\begin{theorem}\label{2.2}
Let $v$ be a type $n$ invariant, or more generally a type $(n,\dots,n)$ invariant of tangles.
If $v$ is invariant under PL isotopy, then it is invariant under $n$-quasi-isotopy.
\end{theorem}

\begin{proof}[Proof. Case $n=0$] The relation of $0$-quasi-isotopy is the same as that of link homotopy.
On the other hand, every type $(0,\dots,0)$ invariant assumes the same values on link homotopic tangles.
This completes the proof of the theorem in the case $n=0$.
\end{proof}

\begin{proof}[Step I (includes the case $n=1$)]
Assume that $n\ge 1$ and that $v$ is invariant under PL isotopy.
It suffices to show that $v^\x$ vanishes on every proper $n$-quasi-embedding $f\:\Theta\to M$ of boundary pattern $\Xi$
which is a singular tangle (i.e.\ its double point is rigid).
Let $x=f(p)=f(q)$ be the double point of $f$.
Let $P_0,\dots,P_n$ and $J_0,\dots,J_n$ be as in the definition of an $n$-quasi-embedding.
We may assume that $\partial J_0=\{p,q\}$.
By passing to small regular neighborhoods we may further assume that $P_1,\dots,P_n$ are 
compact $3$-manifolds (with boundary) and that each $J_i$, $i<n$, lies in the interior of 
an arc $J_i^+$ such that $P_i\cup f(J_i^+)$ lies in $P_{i+1}$ and is null-homotopic in it 
(and also $f^{-1}(P_i)\subset J_i$, as before).
Moreover, it is easy to see that $P_0=\{x\}$ lies in the interior of a PL ball $P_0^+$ which 
meets $f(\overline{\Theta\but J_0})$ in the arc $f(\overline{J_0^+\but J_0})$, lies in $P_1$ and 
is such that the inclusion $P_0^+\cup f(J_0^+)\to P_1$ is null-homotopic.

Since $P_0^+\cup f(J_0)$ is null-homotopic in $P_1$, there exists a homotopy $f_t\:\Theta\to M$ 
keeping $\overline{\Theta\but J_0}$ fixed and keeping $J_0$ within $P_1$, from $f_0=f$ to a map $f_1$ 
sending $J_0$ into the ball $P_0^+$ and such that $f_1(J_0)$ is a small circle bounding a small 
embedded disk in $P_0^+$, whose interior is disjoint from $f(J_0^+\but J_0)$.
Moreover, since $P_0^+$ and $P_1$ are $3$-manifolds, we may assume that $x$ is a rigid 
double point of $f_t$ for each $t\in I$, and apart from this permanent double point, the homotopy $f_t$ 
has only finitely many additional double points that occur at distinct time instants 
$t_1,\dots,t_l\in (0,1)$ and are rigid.
Thus each $f_{t_i}$ is a singular tangle with two double points, and $f_t$ for each 
$t\notin\{t_1,\dots,t_l\}$ is a singular tangle with one double point.

Since the two resolutions $L_+$, $L_-$ of the singular tangle $L:=f_1$ are ambient isotopic, 
every colored finite type invariant, when extended to singular tangles, vanishes on $f_1$.
(This relation is known as ``the one-term relation'' or ``the framing independence relation''
in the theory of finite type invariants, cf.\ \cite{CDM}*{\S4}.)

On the other hand, since $f^{-1}(P_1)$, and hence also each $f_{t_i}^{-1}(P_1)$, lies in the arc $J_1$, 
the two double points of each $f_{t_i}$ are self-intersections of the same component.
Hence every type $(1,\dots,1)$ invariant $u$, when extended to singular tangles, vanishes on each $f_{t_i}$.
Therefore $u^\x(f)=u^\x(f_1)$, but from the one-term relation we also know that $u^\x(f_1)=0$.
This completes the proof of the theorem in the case $n=1$.
\end{proof}

The following lemma can be regarded as a generalization of the one-term relation.

\begin{lemma}\label{2.3} Let $L$ be a singular tangle with $n$ double points, where $n\ge 1$, and suppose 
that all its double points lie in the interior of a PL ball $B$ such that $L^{-1}(B)$ is an arc.
Then for every invariant $u$ of PL isotopy, $u^\x(L)=0$.
\end{lemma}

\begin{proof} Pick any double point of $L$ and consider the two resolutions $L_+$ and $L_-$ of $L$
at this double point.

If $n=1$, then $L_+$ and $L_-$ are non-singular tangles which meet the ball $B$ in possibly distinct local knots 
and agree outside $B$.
Hence they are PL isotopic, and therefore $u^\x(L)=u(L_+)-u(L_-)=0$.

If $n>1$, then, arguing by induction, we may assume that the lemma is established for all smaller values of $n$.
Then $u^\x(L_+)=u^\x(L_-)=0$ by the induction hypothesis, and hence $u^\x(L)=0$.
\end{proof}

\begin{proof}[Step II (includes the case $n=2$)] 
To continue the proof of the theorem for $n\ge 2$, we fix some $i\in\{1,\dots,l\}$
and study the jump of $v^\x$ on the singular tangle $f_{t_i}$.

Since $f^{-1}(P_1)$ lies in $J_1$, so does $f_{t_i}^{-1}(P_1)$, and consequently
both double points of $f_{t_i}$ lie in the graph $G:=f_{t_i}(J_1)$.
Let $g$ be an embedding of the graph $G^+:=f_{t_i}(J_1^+)$ into $P_0^+\cup f(J_1^+\but J_0)$
which restricts to an embedding of $G$ into the ball $P_0^+$, to the identity on 
the two points $f_{t_i}(\partial J_1^+)=f(\partial J_1^+)$ and to a homeomorphism between 
the two arcs $f_{t_i}(\overline{J_1^+\but J_1})$ and the two arcs $f(\overline{J_1^+\but J_0^+})$.
Since $G^+$ and $g(G^+)$ both lie in $P_1\cup f(J_1^+)$, which is in turn null-homotopic in $P_2$,
the inclusion $G^+\to P_2$ is homotopic to $g$ by a homotopy $g_t\:G^+\to P_2$ 
keeping $f(\partial J_1^+)$ fixed.%
\footnote{Indeed, if $X\supset A\subset Y$, the inclusion $A\to Y$ is null-homotopic and the inclusion $A\to X$ 
is a cofibration, then it is easy to see that the inclusion $A\to Y$ extends to a map $X\to Y$.}
The composition $J_1^+\xr{f_{t_i}}f_{t_i}(J_1^+)=G^+\xr{g_t}P_2$ extends by the identity 
on $\overline{\Theta\but J_1^+}$ to a homotopy $f_{t_i;\,t}\:\Theta\to M$ from $f_{t_i;\,0}=f_{t_i}$ to a map $f_{t_i;\,1}$ 
which has precisely two double points, both contained in $P_0^+$, and is such that $f_{t_i;\,1}^{-1}(P_0^+)=J_1$.
Moreover, since $P_0^+$ and $P_2$ are $3$-manifolds, we may assume that the two double points of $f_{t_i}$ 
continue as rigid double points of $f_{t_i;\,t}$ for each $t\in I$, and apart from these two permanent double points, 
the homotopy $f_{t_i;\,t}$ has only finitely many additional double points that occur at distinct time instants 
$t_{i1},\dots,t_{ik_i}\in (0,1)$ and are rigid.
Thus each $f_{t_i;\,t_{ij}}$ is a singular tangle with three double points and $f_{t_i;\,t}$ for each 
$t\notin\{t_{i1},\dots,t_{ik_i}\}$ is a singular tangle with two double points.

Since both double points of $f_{t_i;\,1}$ are contained in $P_0^+$ and $f_{t_i;\,1}^{-1}(P_0^+)$ is the arc $J_1$,
by Lemma \ref{2.3} every colored finite type invariant, when extended to singular tangles, vanishes on $f_{t_i;\,1}$.

On the other hand, since $f^{-1}(P_2)$, and hence also each $f_{t_i;\,t_{ij}}^{-1}(P_2)$, lies in the arc $J_2$, 
the three double points of each $f_{t_i;\,t_{ij}}$ are self-intersections of the same component.
Hence every type $(2,\dots,2)$ invariant $u$, when extended to singular tangles, vanishes on each $f_{t_i;\,t_{ij}}$.
Therefore $u^\x(f_{t_i})=u^\x(f_{t_i;\,1})$, but from Lemma \ref{2.3} we also know that $u^\x(f_{t_i;\,1})=0$.
Thus $u^\x$ vanishes on each $f_{t_i}$.
Consequently $u^\x(f)=u^\x(f_1)$, but from the one-term relation we also know that $u^\x(f_1)=0$.
This completes the proof of the theorem in the case $n=2$.
\end{proof}

The proof of the general case is similar to the case $n=2$ (and if the reader feels that reading the proof of the case $n=2$ 
was enough for her to understand the proof of the general case, she is probably right), but for completeness we include the details.

\begin{proof}[Step III (Case $n\ge 2$)]
Suppose that for each $k=1,\dots,n-1$ we have constructed a collection $C_k$ of homotopies $f_{s_1,\dots,s_{k-1};\,t}$ 
which go through singular tangles with $k$ double points, apart from finitely many critical levels, which are singular
tangles with $k+1$ double points, so that the following conditions hold.
\begin{enumerate}
\item $f_0=f$, and for every $k\in\{1,\dots,n-2\}$ a singular tangle is the initial level $f_{s_1,\dots,s_k;\,0}$
of a homotopy of the form $f_{s_1,\dots,s_k;\,t}\in C_{k+1}$ if and only if it is a critical
level, corresponding to $t=s_k$, of a homotopy of the form $f_{s_1,\dots,s_{k-1};\,t}\in C_k$. 
\item The terminal level $f_{s_1,\dots,s_k;\,1}$ of every homotopy $f_{s_1,\dots,s_k;\,t}\in C_{k+1}$
has all its double points in the ball $P_0^+$ and satisfies $f_{s_1,\dots,s_k;\,1}^{-1}(P_0^+)=J_k$.
\item Every level of every homotopy $f_{s_1,\dots,s_{k-1};\,t}\in C_k$ agrees with $f$ outside $J_{k-1}^+$, 
and sends $J_{k-1}^+$ into $P_k$.
\end{enumerate}
Let us show that the given collections $C_1,\dots,C_{n-1}$ can be extended by a collection $C_n$ so that
the collections $C_1,\dots,C_n$ satisfy the same conditions with $n$ replaced by $n+1$.
 
Suppose that $s_{n-1}$ is a critical time instant of a homotopy $f_{s_1,\dots,s_{n-2};\,t}\in C_{n-1}$.
Let $S=(s_1,\dots,s_{n-1})$.
Thus $f_S$ is a singular tangle with $n$ double points, which agrees with $f$ outside $J_{n-2}^+$ and 
sends $J_{n-2}^+$ into $P_{n-1}$.
Since $f^{-1}(P_{n-1})$ lies in $J_{n-1}$, so does $f_S^{-1}(P_{n-1})$, and consequently
all double points of $f_S$ lie in the graph $G:=f_S(J_{n-1})$.
Let $g$ be an embedding of the graph $G^+:=f_S(J_{n-1}^+)$ into $P_0^+\cup f(\overline{J_{n-1}^+\but J_0})$
which restricts to an embedding of $G$ into the ball $P_0^+$, to the identity on 
the two points $f_S(\partial J_{n-1}^+)=f(\partial J_{n-1}^+)$ and to a homeomorphism between 
the two arcs $f_S(\overline{J_{n-1}^+\but J_{n-1}})$ and the two arcs $f(\overline{J_{n-1}^+\but J_0^+})$.
Since $G^+$ and $g(G^+)$ both lie in $P_{n-1}\cup f(J_{n-1}^+)$, which is in turn null-homotopic in $P_n$,
the inclusion $G^+\to P_n$ is homotopic to $g$ by a homotopy $g_t\:G^+\to P_n$ keeping $f(\partial J_{n-1}^+)$ fixed.
The composition $J_{n-1}^+\xr{f_S}f_S(J_{n-1}^+)=G^+\xr{g_t}P_n$ extends by the identity 
on $\overline{\Theta\but J_{n-1}^+}$ to a homotopy $f_{S;\,t}\:\Theta\to M$ from $f_{S;\,0}=f_S$ to a map $f_{S;\,1}$ 
which has precisely $n$ double points, all contained in $P_0^+$, and is such that $f_{S;\,1}^{-1}(P_0^+)=J_{n-1}$.
Moreover, since $P_0^+$ and $P_n$ are $3$-manifolds, we may assume that the $n$ double points of $f_S$ 
continue as rigid double points of each $f_{S;\,t}$, and apart from these $n$ permanent double points, the homotopy 
$f_{S;\,t}$ has only finitely many additional double points that occur at distinct time instances 
$u_1,\dots,u_k\in (0,1)$ and are rigid.
Thus each $f_{S;\,u_j}$ is a singular tangle with $n+1$ double points and $f_{S;\,t}$ for each 
$t\notin\{u_1,\dots,u_k\}$ is a singular tangle with $n$ double points.
This completes the construction of the desired collection $C_n$.

Now that the collections $C_1,\dots,C_n$ have been constructed, we are ready to evaluate the given invariant $v$
of type $(n,\dots,n)$, which is well-defined up to PL isotopy, on the $n$-quasi-embedding $f$.
For each $k=1,\dots,n$ and each homotopy $f_{S;\,t}\in C_k$ we have $v^\x(f_{S;\,1})=0$ by condition (2)
and Lemma \ref{2.3}.
Since $v$ is a type $(n,\dots,n)$ invariant, and each critical level $f_{S;\,u_j}$ of each homotopy $f_{S;\,t}\in C_n$
has $n+1$ double points, all of them being self-intersections of the same component of $\Theta$ (indeed, they lie
in $f_{S;\,u_j}(J_n)$ by condition (3)), we also get that $v^\x(f_{S;\,0})=v^\x(f_{S;\,1})$.
Thus $v^\x(f_{S;\,0})=0$ for each $f_{S;\,t}\in C_n$.
Hence by condition (1) we obtain that $v^\x$ vanishes on each critical level of each homotopy in $C_{n-1}$.
Proceeding in the same fashion, we eventually obtain that $v(f)=0$.
\end{proof}

The proof of Theorem \ref{2.2} works to prove the following

\begin{theorem}\label{2.2'} Let $v$ be a type $(k_1,\dots,k_m)$ invariant of tangles.
If $v$ is invariant under PL isotopy, then it is invariant under $(k_1,\dots,k_m)$-quasi-isotopy.
\end{theorem}

Theorems \ref{isotopy'} and \ref{2.2'} imply the following strengthening of Corollary \ref{main1''},
which is applied in \cite{M24-2} and \cite{M24-3}.

\begin{corollary} \label{perturbation}
For $i=1,2,\dots$ let $v_i$ be a type $(k_{i,1},\dots,k_{i,m})$ invariant of links which is well-defined up to PL isotopy.
Then each $v_i$ assumes the same value on all sufficiently close $C^0$-approximations of any given topological link
and the extension $\bar v_i$ of $v_i$ by continuity to topological links is an invariant of isotopy 
(which of course assumes the same value on all sufficiently close $C^0$-approximations of any given topological link).

Moreover, if $k_{1,j}=k_{2,j}=\dots$ for some $j$, then the entire collection $(\bar v_1,\bar v_2,\dots)$ is invariant 
under sufficiently small $C^0$-perturbation of the $j^{\text{th}}$ component. 
That is, for every topological link $\L$ there exists an $\epsilon>0$ such that for every topological link $\L'$ 
such that $\L'|_{S^1_j}$ is $\epsilon$-close to $\L|_{S^1_j}$ and $\L'|_{S^1_i}=\L|_{S^1_i}$ for $i\ne j$ we have 
$\bar v_i(\L)=\bar v_i(\L')$ for all $i$.
\end{corollary}

\begin{remark} The proof of Theorem \ref{2.2} can be modified so as to work under a weaker hypothesis, with $n$-quasi-isotopy
replaced by {\it virtual $n$-quasi-isotopy}, which is defined in \cite{MR1}*{end of \S3}.
The modified construction is geometrically more intuitive, but its accurate description would be longer and more complicated.
We give only a rough sketch of the modified construction.
The homotopy $f_t$ is the same as before.
Using the notation from \cite{MR1}*{proof of Lemma 3.1}, each homotopy $f_{t_j,t}$ can be visualized as shifting the arc
$f(I'_j)$ onto the arc $F(I_j)$ and then taking the image of $J_1$ into a ball along the track of the null-homotopy $F$.
Thus $f_{t_j,1}$ still satisfies the hypothesis of Lemma \ref{2.3}, but, unless $I'_j\subset J_0$, the image of $f_{t_j,1}$ 
is now different: instead of being the same as the image of $f$ outside the ball $P_0^+$, it now coincides with the image 
of $f$ only outside a small neighborhood of the arc $f(I'_j)$.
\end{remark}

\begin{remark} The proof of Theorem \ref{2.2}, as well as its modification sketched in the previous remark, can be seen
to yield a stronger conclusion: if tangles $L$ and $L'$ are $n$-quasi-isotopic (or just virtually $n$-quasi-isotopic), then
they are not just indistinguishable by type $(n,\dots,n)$ invariants that are well-defined up to PL isotopy, but also
{\it geometrically $(n,\dots,n)$-equivalent} in the following sense.
Let $\LM$ denote the space of all singular tangles which are link maps.
(The topology is the restriction of the $C^0$ topology on the space of maps, modified so as to ensure that double points
remain rigid under homotopies given by paths in this space.)
Let $\LM_n$ denote the subset of $\LM$ consisting of singular tangles with precisely $n$ double points, and let
$\LM_{(n,\dots,n)}$ denote the subset of $\LM_n$ consisting of singular tangles whose all double points occur on the same component.
On the other hand, let $\LM^\circ$ denote the subset of $\LM$ consisting of those singular tangles $L$ whose all double points lie in 
a ball $B$ such that $L^{-1}(B)$ is an arc.

We call two elements of $\LM_n$ {\it geometrically $k$-equivalent} if they are homotopic within the space $\LM_n\cup\LM_{n+1,k}$, 
where $\LM_{i,0}=\LM_i$ and $\LM_{i,k}$ for $k>0$ consists of those elements of $\LM_i$ that are geometrically 
$(k-1)$-equivalent to an element of $\LM_i\cap\LM^\circ$.
Two elements of $\LM_{(n,\dots,n)}$ are {\it geometrically $(k,\dots,k)$-equivalent} if the same holds with $\LM_{(i,\dots,i)}$ 
in place of $\LM_i$.

Let us start unwrapping these definitions; all double points will be assumed to be self-intersections of components.
\begin{itemize}
\item[0/0] Two tangles (=elements of $\LM_0$) are geometrically $0$-equivalent iff they are homotopic within $\LM_0\cup\LM_1$ --- that is, link homotopic.
\item[1/0] Two singular tangles with precisely $1$ double point each (=elements of $\LM_1$) are geometrically $0$-equivalent iff they are homotopic 
within $\LM_1\cup\LM_2$ --- that is, link homotopic keeping the double point rigid.
\item[0/1] Two tangles (=elements of $\LM_0$) are geometrically $1$-equivalent iff they are homotopic within $\LM_0\cup\LM_{1,1}$, where $\LM_{1,1}$ 
consists of those singular tangles with precisely $1$ double point that can be link homotoped keeping the double point rigid until the double point
gets inside a ball whose preimage is an arc.
\item[(1,1)/0] Two singular tangles with precisely $1$ double point each (=elements of $\LM_{(1,1)}$) are geometrically $0$-equivalent iff they 
are homotopic within $\LM_{(1,1)}\cup\LM_{(2,2)}$ --- that is, link homotopic keeping the double point rigid and without self-intersecting 
the non-singular component.
\end{itemize}

It is shown in the proof of Theorem \ref{2.2} that if tangles $L$ and $L'$ are geometrically $(k,\dots,k)$-equivalent, then they are not
separated by type $(k,\dots,k)$ invariants that are well-defined up to PL isotopy.
The same argument works to show that if $L$ and $L'$ are geometrically $k$-equivalent, then they are not separated by colored type 
$k$ invariants that are well-defined up to PL isotopy.
\end{remark}

\begin{example} \label{weak}
(a) It is easy to see that geometric $1$-equivalence implies weak $1$-quasi-isotopy.%
\footnote{Weak $n$-quasi-isotopy is defined similarly to $n$-quasi-isotopy but with ``null-homotopic'' replaced by
``induces zero homomorphisms on reduced integral homology'' (cf.\ \cite{MR1}, \cite{MR2}, see also \cite{MR1}*{Figure 2(d)}).}
For $2$-component links the converse also holds (see \cite{Mi1}*{Figure 2}).

(b) It is shown in \cite{KL} that $2$-component links are not separated by colored type $1$ invariants if and only if they have 
the same linking number and the same generalized Sato--Levine invariant (see Example \ref{conway-ex} concerning the latter).
However, it seems highly unlikely that these two invariants constitute a complete set of invariants of weak $1$-quasi-isotopy
(see, in particular, \cite{MR1}*{Problem 1.5}, which also makes sense for weak $1$-quasi-isotopy in place of $1$-quasi-isotopy).
In fact they constitute a complete set of invariants of $\frac12$-quasi-isotopy (see Example \ref{conway-ex}). 

(c) Geometric $(1,\dots,1)$-equivalence is clearly the same thing as $1$-quasi-isotopy.
\end{example}

\begin{example} \label{milnor-double}
Let $M_k=(K,Q)$ be the $k$th Milnor link (see Figure \ref{milnor}), where $Q$ is the ``long'' component,
and let $M^W_k$, $k\ge 2$, be a Whitehead double of $M_k$ along $Q$.
Thus $M^W_k=(K,Q')$, where $Q'$ is the image of the Whitehead curve in $S^1\x D^2$ under some homeomorphism of $S^1\x D^2$ 
with a regular neighborhood of $Q$.
Unclasping the clasp of $Q'$ yields a $1$-quasi-isotopy $h_t$ from $M^W_k$ to the unlink.
It is not a $2$-quasi-isotopy (not even a weak $2$-quasi-isotopy), since $M_k$ is not a boundary link, as detected by Cochran's 
derived invariants (see \cite{MR2}*{\S3 and \S1.2}).
It is natural to conjecture that $M^W_k$ is not $2$-quasi-isotopic, nor even weakly $2$-quasi-isotopic to the unlink.
However, it is not hard to see that $h_t$ yields a geometric $(k,\dots,k)$-equivalence between $M^W_k$ and the unlink.
\end{example}

\section{Locally additive invariants and $C_n$-equivalence} \label{cn-equivalence}

A tangle is called {\it totally split} if its components are contained in pairwise disjoint $3$-balls.
If $K$ and $L$ are $m$-component links and $K$ is totally split, then the (componentwise) connected sum
$L\#K$ is easily seen to be well-defined.
Also, for any string links $L$ and $L'$ their connected sum $L\# L'$ is well defined.

We call an invariant $v$ of $m$-component (string) links {\it locally additive} if $v(L\# K)=v(L)+v(K)$ 
whenever $K$ is totally split.
This implies in particular that $v$ vanishes on the trivial (string) link.

For every link $L$ in $S^3$ (respectively, for every string link) there is a unique, 
up to ambient isotopy, totally split link in $S^3$ (respectively, totally split string link) $K_L$ whose 
components are ambient isotopic to the respective components of $L$.

\begin{lemma} \label{LA-lemma0}
Let $A$ and $B$ be $m$-component links in $S^3$ (respectively, string links).
Then $A$ and $B$ are not separated by type $n$ invariants that are well-defined up to PL isotopy
if and only if $A$ and $B$ are PL isotopic to (string) links that are not separated by locally additive type $n$ invariants.
\end{lemma}

\begin{proof} {\it ``If''.} It suffices to observe that if an invariant is well-defined up to PL isotopy, then it is locally additive.

{\it ``Only if''.} Clearly, $A':=A\#K_B$ and $B':=B\#K_A$ are PL isotopic to $A$ and $B$ (respectively).
Let $v$ be a locally additive type $n$ invariant of $m$-component (string) links.
Then $\bar v$, defined by $\bar v(L)=v(L)-v(K_L)$, is a type $n$ invariant which is well-defined up to PL isotopy.
(Let us note that $v(K_L)=(v^\kappa_1)^\lambda_1(L)+\dots+(v^\kappa_m)^\lambda_m(L)$ in the notation of Lemma \ref{knot-link}.)
We have $\bar v(A')=\bar v(A)=\bar v(B)=\bar v(B')$.
Since $K_{A'}=K_{B'}$, we obtain that $v(A')=\bar v(A')+v(K_{A'})=\bar v(B')+v(K_{B'})=v(B')$.
\end{proof}

\begin{theorem} \label{locally-additive2}
If two $m$-component string links are separated by invariants of type $n$, then they are separated by locally additive 
invariants of type $r$ for some $r=r(m,n)$.
\end{theorem}

Theorem \ref{locally-additive2} will be proved in the present section by building on Habiro's proof 
\cite{Hab}*{Theorem 6.18} of T. Stanford's theorem \cite{Sta2}*{Theorem 2.43} (cf.\ \cite{Hab}*{Remark 3.19})
that two knots are separated by type $n$ invariants if and only if they are separated by additive type $n$ invariants.
(The rational version of this theorem was originally proved by Gusarov \cite{Gu1}*{Theorem 5.2}; see also
\cite{TaY} for an alternative proof.)
Habiro's proof is in turn based on the study of $C_n$-equivalence. 

\begin{figure}[h]
\includegraphics[width=\linewidth]{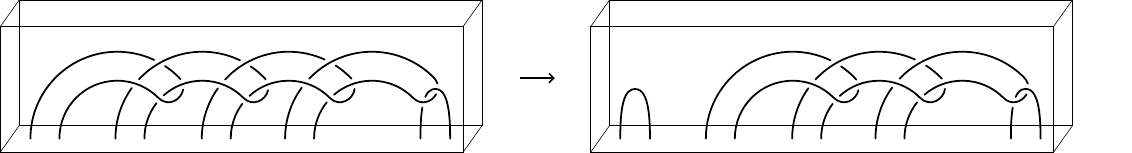}
\caption{The $C_n$-move for $n=4$.}
\label{cn-move}
\end{figure}

Two tangles $L,L\:\Theta\to M$ are said to be related by a {\it $C_n$-move} if they agree outside a $3$-ball $B$, and intersect
$B$ along the $(n+1)$-strand tangles $T,T'\:(n+1)I\to I^3$, which are shown in Figure \ref{cn-move} in the case $n=4$ 
and are similarly drawn for arbitrary $n\ge 1$.
If all the $n+1$ strands of $T$ belong to the same component of $L$, then the $C_n$-move is said to be a {\it self $C_n$-move}.
$L$ and $L'$ are said to be {\it (self) $C_n$-equivalent} if they are related by a sequence of (self) $C_n$-moves and ambient isotopies.

\begin{remark}\label{Brunnian} It is easy to see that the $C_n$-move is ``Brunnian'' in the sense that it becomes trivial (i.e.\ is realized 
by an ambient isotopy) if for any $i\in\{1,\dots,n+1\}$, the $i$th strand is erased from the picture (both on the left and on the right).
\end{remark}

\begin{remark} \label{cn-definition}
The literature contains at least four different definitions of a $C_n$-move, all leading to the same notion of $C_n$-equivalence:
\begin{enumerate}
\item Habiro's second definition \cite{Hab}*{\S3.2 (see also \S7)}): a ``$C_n$-move'' on $L$ is the operation of taking a band 
connected sum of $L$ with an arbitrary $n$-tuply iterated untwisted Bing doubling $\Lambda$ of the Hopf link,%
\footnote{An $n$-tuply iterated untwisted Bing doubling of a $2$-component link is defined in \cite{Co3}*{\S7.4}.
It is determined by a rooted unitrivalent tree with $n$ trivalent vertices.
The two (non-iterated) Bing doublings of the Hopf link are easily seen to be ambient isotopic to the Borromean rings and
in particular to each other.
This implies that an iterated untwisted Bing doubling of the Hopf link depends only on the tree and not on its root.}
along bands which intersect a ball containing $\Lambda$ in a specified way.
\item The operation of taking a band connected sum of $L$ with Milnor's $(n+1)$-component Brunnian link $M_{n+1}$, along bands which 
intersect a ball containing $M_n$ in a specified way.
This is a special case of Habiro's second definition and is called ``band sum with one-branched $C_n$-chord'' in \cite{MiY}.
This definition is used for instance in \cite{CDM}.
\item Habiro's original definition (from his unpublished master's thesis, written in Japanese): a ``$C_n$-move'' is an arbitrary 
$n$-tuply iterated Bing doubling of a crossing change%
\footnote{The two (non-iterated) Bing doublings of the crossing change are easily seen to be equivalent to the $\Delta$-move and
in particular to each other.
This implies that an iterated Bing doubling of the crossing change is fully determined by a unitrivalent tree and does not
depend on its root.}
(for the details see the introduction of \cite{MiY}).
\item $C_n$-move in the sense of Figure \ref{cn-move}.
This is a special case of Habiro's original definition and is called a ``one-branched $C_n$-move'' in \cite{MiY} and \cite{TaY}.
(But in later papers it is often called just a ``$C_n$-move''.)
\end{enumerate}
A proof that (3) and (4) lead to the same notion of $C_n$-equivalence can be found in \cite{TaY}*{Lemma 2.2}.
For a proof that (1) and (3) lead to the same notion of $C_n$-equivalence see \cite{TaY}*{Lemma 3.6}.
As noted in \cite{MiY}*{Lemma 2.1}, the latter argument also shows that (2) and (4) lead to the same 
notion of $C_n$-equivalence.
Moreover, as noted in \cite{MiY}*{Lemma 2.1} and in \cite{TaY}*{third line of the proof of Lemma 2.2},
the arguments in \cite{TaY}*{proofs of Lemmas 2.2 and 3.6} preserve the set of components of $L$ that are involved 
in the moves, so in particular they apply to self $C_n$-moves in place of $C_n$-moves.

It should also be noted that Gusarov's definition of an $n$-variation \cite{Gu2} is close to the definition (1) of 
a $C_n$-move.
Gusarov's previous definition of an $n$-equivalence \cite{Gu1}, which developed Ohyama's notion of $n$-triviality,
is not directly related to either of the definitions (1)--(4).
\end{remark}

\begin{remark} It is well-known that a $C_{n+1}$-move is realized by a sequence of $C_n$-moves (see \cite{CDM}*{\S14.2.1}).
\end{remark}

\begin{proposition}[Gusarov \cite{Gu2}*{10.3}, Habiro \cite{Hab}*{6.8}] \label{cn-eq} 
If two tangles are $C_{n+1}$-equivalent, then they are not separated by invariants of type $n$.
\end{proposition}

Although the literature also contains other proofs of Proposition \ref{cn-eq} (see in particular \cite{CDM}*{\S14.2.3}
for a comparatively short one), the following very simple proof does not seem to appear anywhere.

\begin{proof}
Let $v$ be an invariant of tangles with a fixed boundary pattern and $v^\x$ be its standard extension to singular tangles
with the same boundary pattern.
Figure \ref{cn-move2} shows a generic homotopy with two double points between two tangles $L$ and $L'$ related by a $C_n$-move, $n\ge 2$.
The total change of $v$ under this homotopy equals the difference of the values of $v^\x$ on the two singular tangles 
$L_\x$ and $L'_\x$ which occur in it.
Clearly $L_\x$ and $L'_\x$ are homotopic, so if $v$ is a type $1$ invariant, then $v^\x(L_\x)=v^\x(L'_\x)$ and hence $v(L)=v(L')$.

\begin{figure}[h]
\includegraphics[width=\linewidth]{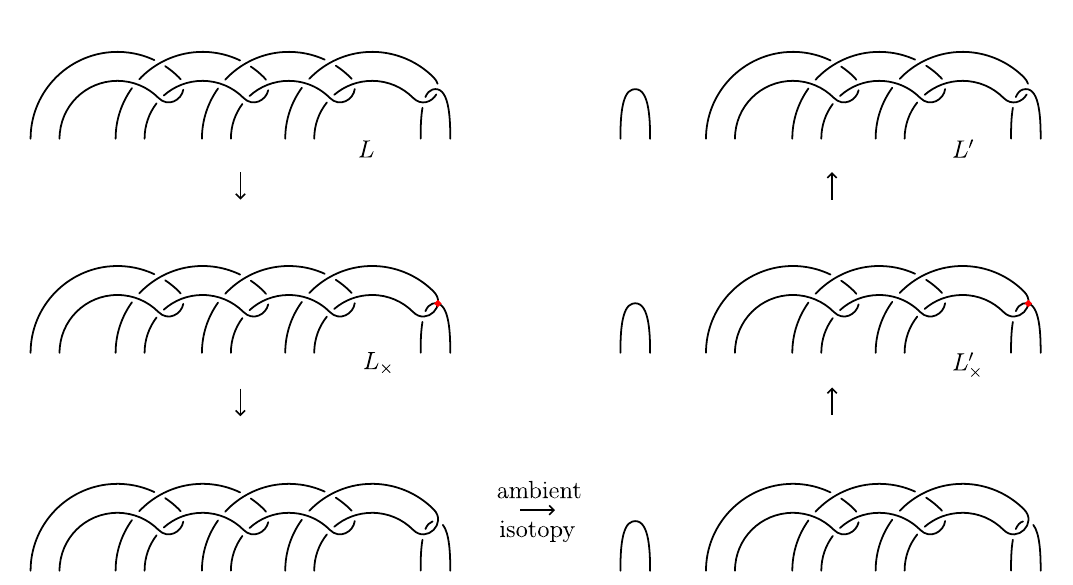}
\caption{Type $1$ invariants are invariant under $C_n$-move for $n\ge 2$.}
\vspace{30pt}\label{cn-move2}
\includegraphics[width=\linewidth]{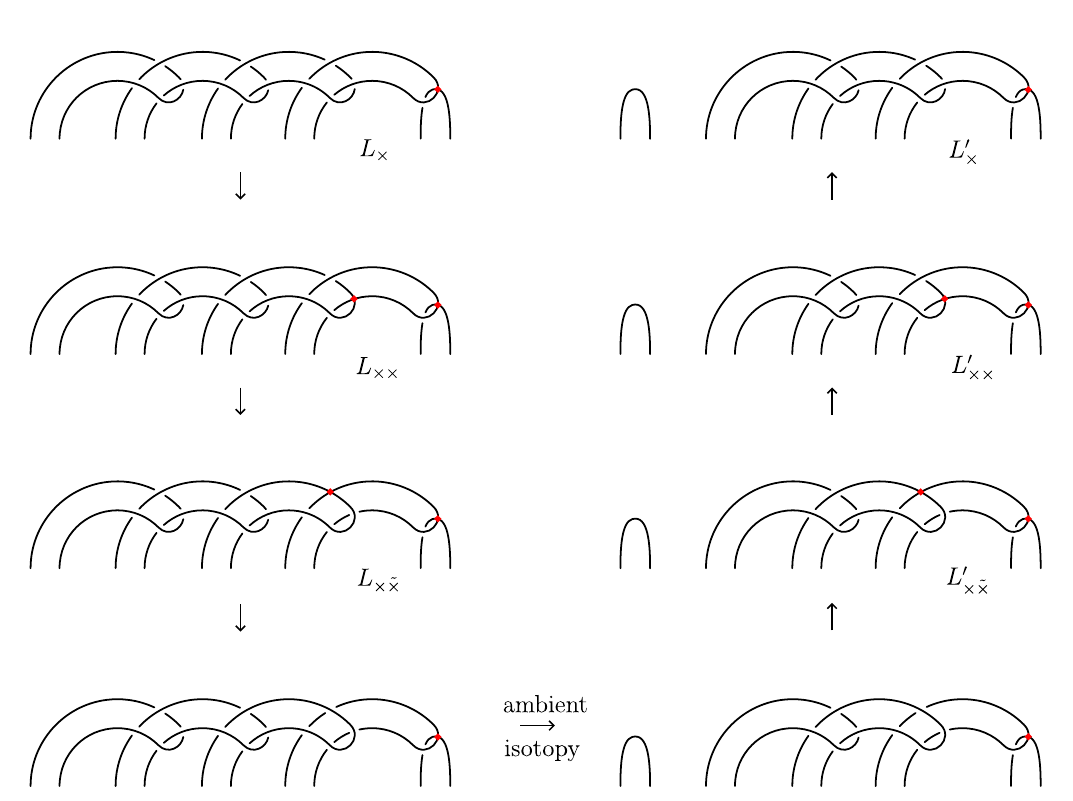}
\caption{Type $2$ invariants are invariant under $C_n$-move for $n\ge 3$.}
\vspace{-60pt}
\label{cn-move3}
\end{figure}

For $n\ge 3$ Figure \ref{cn-move3} shows a generic homotopy with four new double points between the singular tangles $L_\x$ and $L'_\x$.
The total change of $v^\x$ under this homotopy is the difference of its values on the doubly singular tangles $L_{\x\x}$ and $L'_{\x\x}$
plus the difference of its values on the other two doubly singular tangles $L_{\x\tilde\x}$ and $L'_{\x\tilde\x}$.
Clearly the former two are homotopic to each other and the latter two are homotopic to each other; so if $v$ is a type $2$ invariant, then 
$v^\x(L_{\x\x})=v^\x(L'_{\x\x})$ and $v^\x(L_{\x\tilde\x})=v^\x(L'_{\x\tilde\x})$.
Then $v^\x(L_\x)=v^\x(L'_\x)$ and hence $v(L)=v(L')$.

Proceeding in the same fashion, we will eventually obtain that if $v$ is a type $n-1$ invariant, then $v(L)=v(L')$.
\end{proof}

\begin{remark} \label{self-colored}
The proof of Proposition \ref{cn-eq} also shows that if two links or string links are self $C_{n+1}$-equivalent, 
then they are not separated by invariants of type $(n,\dots,n)$.
\end{remark}

Gusarov and Habiro also proved that the converse to Proposition \ref{cn-eq} holds for knots 
\cite{Hab}*{Theorem 6.18}, \cite{Gu2}*{Remark 10.5} (the rational version was originally proved by Gusarov \cite{Gu1}*{Theorem 5.2};
alternative proofs appear in \cite{Sta2}, \cite{Fu}, \cite{Ya06}) but does not hold for links of two components \cite{Gu2}*{Remark 10.8}
and of more than two components \cite{Hab}*{Proposition 7.4}.
However, they conjectured that it holds for string links \cite{Hab}*{Conjecture 6.13}, \cite{Gu2}*{Remark 10.7},
and this conjecture has been verified for $n\le 4$ \cite{MY}.
Gusarov proved a version of this conjecture for partially defined invariants \cite{Gu2}*{Theorem 10.4} and
G. Massuyeau proved the following rational version of this conjecture: If two string links $L$, $L'$ are not separated
by rational invariants of type $n$, then there exists a $k$ such that $\underbrace{L\#\dots\#L}_k$ and 
$\underbrace{L'\#\dots\#L'}_k$ are $C_{n+1}$-equivalent \cite{Mass}*{Theorem 1.1} (see also \cite{CDM}*{12.6.3}).
He also proved the following:

\begin{theorem}[Massuyeau \cite{Mass}*{Corollary 1.3}] \label{massuyeau}
There exists an $r=r(m,n)$ such that if two $m$-component string links are not separated
by invariants of type $r$, then they are $C_n$-equivalent.
\end{theorem}

It appears to be unknown whether Theorem \ref{massuyeau} holds for (closed) links (compare \cite{HM} and 
in particular Problem 5.4 in there).

In fact, Massuyeau's results in \cite{Mass} are stated in a more general setting of homology cylinders, 
and his proofs are also in that setting.
For our purposes, it is preferable to understand the proof of Theorem \ref{massuyeau} in the original 
language of string links.
To this end we look more closely at Habiro's paper \cite{Hab}.

A slightly weaker statement, asserting that if two string links are not separated by finite type invariants 
then they are $C_k$-equivalent for all $k$, is also found in Habiro's paper \cite{Hab}*{Corollary 6.12}
(this result was also announced by Gusarov \cite{Gu2}*{Remark 10.7} but apparently not written up because of his premature death).
However, it is deduced from \cite{Hab}*{Theorem 6.11}, whose proof in \cite{Hab} contains an error.
Namely, the author claims: ``Since $J$ is the augmentation ideal of the group ring of the nilpotent group $\L_1(\Sigma,n)/C_k$, we
have $\bigcap_{l=1}^\infty J^l=\{0\}$.''
But it is not true that if $J$ is the augmentation ideal of the (integral) group ring of a nilpotent group $G$,
then $\bigcap_{l=1}^\infty J^l=\{0\}$.
For instance, this statement is false for $G=\Z/pq$, where $p$ and $q$ are distinct primes (see \cite{Bu}*{proof of Theorem 1}).
Also, already for $G=\Z$ none of the powers $J^l$ is zero itself, which is why Habiro's assertion \cite{Hab}*{Corollary 6.12}
is weaker than Theorem \ref{massuyeau}.
In the case where $\Sigma$ is a disk, Habiro's group $\L_1(\Sigma,n)/C_k$ is finitely generated \cite{Hab}*{Theorem 5.4(3)}, 
and in this case we will show that the error can be corrected, and moreover the result can be improved to Theorem \ref{massuyeau},
by using the following lemma.

\begin{lemma}[Hartley {[see below]}; Massuyeau \cite{Mass}*{Corollary 4.8}] \label{dim-subgroup} Let $G$ be a finitely generated nilpotent group.
Let $d_nG=G\cap(1+\Delta^n)$, where $\Delta$ is the augmentation ideal of $\Z G$.
Then $d_nG=1$ for some finite $n$.
\end{lemma}

It is easy to see that $\gamma_nG\subset d_nG$ for all $n$ and it is known that $d_nG=\gamma_nG$ for $n\le 3$ 
(see \cite{Pa}*{V.5.3 and V.5.10}).
It had been an open problem in group theory for about 30 years whether this holds for all $n$, until E. Rips solved it in 1972;
it is now known that for each $n\ge 4$ there exists a finitely presented group $G$ with $\gamma_nG=1$ and $d_nG\ne 1$ 
(see \cite{MiP}*{2.11}); also there exists a finitely presented $G$ with $\gamma_5G=1$ and $d_6G\ne 1$ \cite{MiP}*{2.15}.
However, $\gamma_nG$ always has a finite index in $d_nG$ and this index is bounded above 
by a certain explicit function of $n$ (see \cite{MiP}*{2.17}).
According to \cite{MiP}*{Problem 2.51}, it is unknown whether Lemma \ref{dim-subgroup} 
holds when $G$ is not finitely generated.

\begin{remark} (a) A short proof of Lemma \ref{dim-subgroup} for finite $G$ appears in \cite{Bu}*{Corollary~ 1}.

(b) It is also easy to prove that $\bigcap_n d_nG=1$ under the hypothesis of Lemma \ref{dim-subgroup}.
Indeed, suppose that $\bigcap_{n=1}^\infty d_nG$ contains an element $g\ne 1$.
Since $G$ is finitely generated and nilpotent, it is polycyclic and hence residually finite 
(see \cite{Rob}*{5.2.18 and 5.4.17}).
Thus $G$ admits a homomorphism $f$ onto a finite group $Q$ such that $f(g)\ne 1$.
Clearly, $Q$ is nilpotent (see \cite{Hall}*{10.3.1}) and $f(g)\in\bigcap_{n=1}^\infty d_nQ$, which contradicts 
the case of Lemma \ref{dim-subgroup} where $G$ is finite \cite{Bu}*{Corollary~ 1}.

(c) Lemma \ref{dim-subgroup} can be seen to follow from a result of Hartley \cite{Hart'}.
In more detail, let $T$ be the torsion subgroup of $G$.
Since $G$ is finitely generated and nilpotent, it contains no infinitely ascending chains of subgroups 
(see \cite{Rob}*{5.2.18}), and it follows that $T$ is finitely generated.
Since $T$ is a finitely generated nilpotent torsion group, it is finite (see \cite{Rob}*{5.2.18 or 5.4.11}).
Since $T$ is a finite nilpotent group, by the case of Lemma \ref{dim-subgroup} where $G$ is finite 
\cite{Bu}*{Corollary~ 1} $d_m T=1$ for some finite $m$.
Since $G$ is nilpotent, by Hartley's theorem \cite{Hart'}*{Corollary A1} this implies that $d_n G=1$ 
for some finite $n$.
\end{remark}

\begin{proof}[Alternative proof of Theorem \ref{massuyeau}] This is a correction of \cite{Hab}*{proof of Theorem 6.11}.

Let us fix an $m>1$, and let $\L$ be the monoid of $m$-component string links.
Since the relation of $C_k$-equivalence, viewed as a subset of $\L\x\L$, is a submonoid of $\L\x\L$, 
the quotient $G_k$ of $\L$ by the relation of $C_k$-equivalence is a monoid.
Moreover, the following was shown by Habiro (that $G_k$ is a group, and that it is nilpotent was also 
shown by Gusarov \cite{Gu2}*{9.2, 9.4(3)}):

\begin{lemma}[Habiro \cite{Hab}*{5.4(1,3)}] \label{fg-nilpotent} $G_k$ is a finitely generated nilpotent group.
\end{lemma}

It is easy to see that the group ring $\Z G_k$ is isomorphic to the quotient $\Z\L/W_k$, where $W_k$ is 
the two-sided ideal of the monoid ring $\Z\L$ generated by all elements of the form $L-L'$, where $L$ 
is $C_k$-equivalent to $L'$.

On the other hand, let $V_k$ be the two-sided ideal of $\Z\L$ generated by all singular links with $k$ double points, where a singular link $L$
with $k+1$ double points is identified with an alternating sum of links by the formula (\ref{vassiliev-extension'}).%
\footnote{Thus $V_k$ is $J_k(D^2,m)$ in Habiro's notation; $W_k$ is his $J_{k,1}(D^2,m)$ and $\L$ is his $\L(D^2,m)=\L_1(D^2,m)$.}
Thus a type $k$ invariant on $m$-component string links with values in an abelian group $A$ can be thought of as a homomorphism
$\Z\L\to A$ which vanishes on $V_{k+1}$.

Proposition \ref{cn-eq} implies that $W_k\subset V_k$ (actually it is obtained in \cite{Hab} as a consequence of this inclusion).
It is easy to see that $W_1=V_1$.
Moreover, $W_1$ coincides also with the augmentation ideal $\Delta$ of $\Z\L$, that is, the kernel of the augmentation homomorphism 
$\epsilon\:\Z\L\to\Z$, which is given by $\epsilon(\L)=1$.
(Indeed, $\Delta$ is easily seen to be generated by elements of the form $L_0-L_1$, where $L_0,L_1\in\L$, and if $L_t$ is 
a generic homotopy from $L_0$ to $L_1$, and $L_{t_1},\dots,L_{t_k}$ are its critical levels, which are singular links
with one point each, then $L_0-L_1=\epsilon_1L_{t_1}+\dots+\epsilon_kL_{t_k}$ for some signs $\epsilon_i\in\{1,-1\}$.)
In the direction of the reverse inclusion ``$V_k\subset W_k$'', Habiro shows, in particular:

\begin{lemma}[Habiro \cite{Hab}] \label{habiro-main} $V_{q(n-1)}\subset\Delta^q+W_n$.
\end{lemma}

\begin{proof} This is a special case of \cite{Hab}*{Proposition 6.10}, which is seen by using that if $k_1+\dots+k_l=q(n-1)$, 
then either $l\ge q$ or some $k_i\ge n$.
\end{proof}

Let $\Delta_n$ be the augmentation ideal of $\Z G_n$.
Since the projection $p_n\:\Z\L\to\Z G_n$ commutes with the two augmentation homomorphisms, $p_n(\Delta)\subset\Delta_n$.
Since $\ker p_n=W_n$, we get that $p_n(\Delta^l+W_n)\subset\Delta_n^l$ for all $l$.
On the other hand, by Lemma \ref{dim-subgroup} there exists a $q=q(m,n)$ such that $(1-G_n)\cap\Delta_n^q=0$.
Therefore $(1-\L)\cap (\Delta^q+W_n)\subset\ker p_n=W_n$.
Since $V_{q(n-1)}\subset\Delta^q+W_n$, we get that $(1-\L)\cap V_{q(n-1)}\subset(1-\L)\cap(\Delta^q+W_n)\subset W_n$.

Let $r=q(n-1)-1$.
We may assume that $q,n\ge 2$, whence $r\ge n-1$.
Suppose that $L,L'\in\L$ are not separated by invariants of type $r$.
Thus $L-L'\in V_{r+1}$.
Let $\bar L\in\L$ be a string link which is inverse to $L$ modulo $C_{r+1}$-equivalence.
Since $V_{r+1}$ is a right ideal, $L\bar L-L'\bar L\in V_{r+1}$.
On the other hand $L\bar L$ is $C_{r+1}$-equivalent to $1\in\L$, and therefore $1-L\bar L\in W_{r+1}\subset V_{r+1}$.
Then $1-L'\bar L\in V_{r+1}$.
Consequently $1-L'\bar L\in (1-\L)\cap V_{r+1}\subset W_n$.
Since $W_n$ is a right ideal, $L-L'\bar LL\in W_n$.
On the other hand $\bar LL$ is $C_{r+1}$-equivalent to $1\in\L$, and therefore $\bar LL-1\in W_{r+1}\subset W_n$.
Since $W_n$ is a left ideal, $L'\bar LL-L'\in W_n$.
Thus $L-L'\in W_n$, and so $L$ is $C_n$-equivalent to $L'$.
\end{proof}

\begin{theorem} \label{locally-additive} There exists an $r=r(m,n)$ such that if two $m$-component string links are not separated
by locally additive invariants of type $r$, then they are $C_n$-equivalent.
\end{theorem}

Let us note that $L\# K=K\# L$ if $K$ and $L$ are $m$-component string links and $K$ is totally split.
We will call this property the {\it local commutativity} of string links.
It is proved similarly to the usual proof that the monoid of knots is commutative (see \cite{BZ}*{Figure 7.3} or \cite{Fox-qt}*{p.\ 140}).

\begin{proof} This is an elaboration on the previous proof (the alternative proof of Theorem \ref{massuyeau}), 
additionally employing some ideas from \cite{Hab}*{proof of Theorem 6.18}.
We will use the notation of the previous proof, except that $q$ and $r$ will be chosen differently.

Let $\K$ be the submonoid of $\L$ consisting of all totally split $m$-component string links.
Clearly, an invariant $v\:\L\to A$ is locally additive if and only if its additive extension $\bar v\:\Z\L\to A$
vanishes on all elements of the form $L\#K-L-K$, where $L\in\L$ and $K\in\K$.
This is equivalent to saying that $\bar v$ vanishes on $1\in\L$ and on all elements of the form $(1-L)(1-K)$,
where $L\in\L$ and $K\in\K$.
Since elements of the form $1-L$, where $L\in\L$, additively generate the augmentation ideal $\Delta$ of $\Z\L$, 
and elements of the form $1-K$, where $K\in\K$, additively generate the augmentation ideal $\Lambda$ of $\Z\K$,
elements of the form $(1-L)(1-K)$ additively generate $\Delta\Lambda$.
Thus $v$ is locally additive if and only if $\bar v$ vanishes on $1\in\Lambda$ and on $\Delta\Lambda$.
Let us note that $\Lambda$ is not an ideal of $\Z\L$, but $\Delta\Lambda$ is a two-sided ideal of $\Z\L$.
(Indeed, since $\Delta$ a left ideal, so is $\Delta\Lambda$; similarly, $\Lambda\Delta$ is a right ideal, but
$\Lambda\Delta=\Delta\Lambda$ due to the local commutativity of string links.)

Let $H_n$ be the image of $\K$ in $G_n$, and let $\Lambda_n$ be the augmentation ideal of $\Z H_n$.
It equals $p_n(\Lambda)$ since both are additively generated by elements of the form $1-h$, where $h\in H_n$.
It follows that $p_n(\Delta\Lambda)=\Delta_n\Lambda_n$, and consequently 
$\Z\L/(W_n+\Delta\Lambda)=\Z G_n/(\Delta_n\Lambda_n)$.
Our next goal is to get an explicit description of this quotient ring.

By the local commutativity $H_n$ lies in the center of $G_n$, and in particular it is a normal subgroup.
On the other hand, the totally split string link $K_L$ corresponding to a given string link $L$ may be 
described so as be well-defined not just up to ambient isotopy, but as a specific string link.
Consequently a self-$C_n$-move on $L$ induces a self-$C_n$-move on $K_L$; whereas a $C_n$-move on $L$
which is not a self-$C_n$-move induces an ambient isotopy on $K_L$ (see Remark \ref{Brunnian}).
Thus the $C_n$-equivalence class of $K_L$ is determined by the $C_n$-equivalence class of $L$.
This yields an epimorphism $G_n\to H_n$, which restricts to the identity on $H_n$.
Its kernel $Q_n$ is a normal subgroup of $G_n$ such that $Q_nH_n=G_n$ and $Q_n\cap H_n=1$.
Hence $G_n=Q_n\x H_n$.

Let $R_n$ be the ring whose additive group is $(\Z Q_n)\oplus H_n$, with multiplication given by
$(r_1,h_1)\cdot(r_2,h_2)=(r_1r_2,\,h_1^{\epsilon(r_2)}h_2^{\epsilon(r_1)})$.
(Here $H_n$ is written multiplicatively, even though it is abelian. 
It straightforward to verify that the multiplication is associative and distributive with respect to the addition.
Clearly, the multiplicative identity is $(1,1)$.)
A homomorphism of the additive groups $\phi_n\:\Z G_n\to R_n$ is given on the additive generators by 
$\phi_n(qh)=(q,h)$ for any $q\in Q_n$ and $h\in H_n$ and is easily seen to be a ring homomorphism
(it suffices to check its multiplicativity on the additive generators).
It can be described in general by
$\phi_n\big(m_1q_1h_1+\dots+m_kq_kh_k\big)=(m_1q_1+\dots+m_kq_k,\,h_1^{m_1}\cdots h_k^{m_k})$,
where $h_i\in H_n$, $q_i\in Q_n$ and $m_i\in\Z$.

It is easy to see that $\ker\phi_n$ contains all elements of the form $qh'h-qh'-h+1$  
where $q\in Q_n$ and $h',h\in H_n$.
These are the same as elements of the form $(g-1)(h-1)$, where $g\in G_n$ and $h\in H_n$.
Hence $\ker\phi_n$ contains $\Delta_n\Lambda_n$.
To prove the reverse inclusion it suffices to show that the quotient map $\Z G_n\to\Z G_n/(\Delta_n\Lambda_n)$
factors though $\phi_n$.
The relations $(h-1)(h'-1)=0$ for $h,h'\in H_n$ can be rewritten as $h+h'=hh'+1$ and imply
$m_1h_1+\dots+m_kh_k=h_1^{m_1}\cdots h_k^{m_k}+(m_1+\dots+m_k-1)$ for $h_i\in H_n$ and $m_i\in\Z$.
The relations $(q-1)(h-1)=0$ for $q\in Q_n$ and $h\in H_n$ can be rewritten as $qh=q+h-1$ and imply
$m_1q_1h_1+\dots+m_kq_kh_k=m_1q_1+\dots+m_kq_k+m_1h_1+\dots+m_kh_k-(m_1+\dots+m_k)=
m_1q_1+\dots+m_kq_k-1+h_1^{m_1}\cdots h_k^{m_k}$ for $q_i\in Q_n$, $h_i\in H_n$ and $m_i\in\Z$.
In particular, we get $(m_1q_1+\dots+m_kq_k)h=m_1q_1+\dots+m_kq_k-1+h^{m_1+\dots+m_k}$, or 
$rh=r-1+h^{\epsilon(r)}$ for $r=m_1q_1+\dots+m_kq_k\in\Z Q_n$ and $h\in H_n$.
Hence given $r_1,r_2\in\Z Q_n$ and $h_1,h_2\in H_n$, we obtain 
$(r_1-1+h_1)(r_2-1+h_2)=(r_1-1)(r_2-1)+r_1+r_2-2+h_1^{\epsilon(r_2)}+h_2^{\epsilon(r_1)}-h_1-h_2+h_1h_2=
r_1r_2-1+h_1^{\epsilon(r_2)}h_2^{\epsilon(r_1)}$.
If we introduce the notation $(r,h)=r-1+h$ for $r\in\Z Q_n$ and $h\in H_n$,
this becomes $(r_1,h_1)\cdot(r_2,h_2)=(r_1r_2,\,h_1^{\epsilon(r_2)}h_2^{\epsilon(r_1)})$.
This shows that the ring $R_n$ is a quotient of $\Z G_n$ by an ideal contained in $\Delta_n\Lambda_n$,
with $\phi_n\:\Z G_n\to R_n$ being the quotient map.

Thus we have proved that $\ker\phi_n=\Delta_n\Lambda_n$, and hence the kernel of the composition $\Z\L\xr{p_n}\Z G_n\xr{\phi_n}R_n$ 
equals $W_n+\Delta\Lambda$.
Let $\Gamma_n$ be the augmentation ideal of $\Z Q_n$.
Using the identity $qh-1=(q-1)+(h-1)+(q-1)(h-1)$ it is easy to see that 
$\Delta_n^l\subset\Gamma_n^l+\Delta_n\Lambda_n$ for $l\ge 2$.
Therefore $\phi_n(\Delta_n^l)\subset\Gamma_n^l$ for $l\ge 2$.
Hence $\phi_n p_n(\Delta^l+W_n+\Delta\Lambda)=\phi_n p_n(\Delta^l)\subset\Gamma_n^l$ for $l\ge 2$.
By Lemma \ref{fg-nilpotent} $Q_n$ is nilpotent (as a subgroup of $G_n$) and finitely generated 
(as a quotient of $G_n$); hence by Lemma \ref{dim-subgroup} there exists a $q=q(m,n)\ge 2$ such that 
$(1-Q_n)\cap\Gamma_n^q=0$.
This implies that $\phi_n(1-G_n)\cap\Gamma_n^q=0$.
(Indeed, given any $q\in Q_n$ and $h\in H_n$, if $\phi_n(1-qh)=(1-q,\,1/h)$ lies in $\Gamma_n^q\subset\Z Q_n$, 
then it equals $1-q$ and hence lies in $1-Q_n$.)
Therefore $(1-\L)\cap (\Delta^q+W_n+\Delta\Lambda)\subset\ker(\phi_n p_n)=W_n+\Delta\Lambda$.
On the other hand, by Lemma \ref{habiro-main} $V_{q(n-1)}\subset\Delta^q+W_n$.
Hence $(1-\L)\cap (V_{q(n-1)}+\Delta\Lambda)\subset(1-\L)\cap(\Delta^q+W_n+\Delta\Lambda)\subset W_n+\Delta\Lambda$.

Let $r=q(n-1)-1$.
Suppose that all locally additive type $r$ invariants take the same values on $L,L'\in\L$.
Then $L-L'$ lies in the subgroup of the additive group of $\Z\L$ generated by $1$ and by the ideal 
$V_{r+1}+\Delta\Lambda$.
Since $\epsilon$ vanishes on $L-L'$ and on $V_{r+1}+\Delta\Lambda$, but not on $1$, it follows that 
$L-L'\in V_{r+1}+\Delta\Lambda$.
Arguing like in the previous proof, but now using that $(1-\L)\cap (V_{r+1}+\Delta\Lambda)\subset W_n+\Delta\Lambda$,
we obtain that $L-L'\in W_n+\Delta\Lambda$.
Then $\phi_n p_n(L)=\phi_n p_n(L')$.
But $\phi_n$ is injective on $G_n$ (by the definition).
Hence $p_n(L)=p_n(L')$.
Therefore $L$ is $C_n$-equivalent to $L'$.
\end{proof}

\begin{proof}[Proof of Theorem \ref{locally-additive2}]
By Proposition \ref{cn-eq} and Theorem \ref{locally-additive}.
\end{proof}

\section{Proof of Theorem \ref{rolfsen}} \label{proof-rolfsen}

It does not seem to be easy to prove (or disprove) that for each $n$ there exists an $r=r(n)$ such that if two links in $S^3$ 
are (strongly) $r$-quasi-isotopic, then they are ambient isotopic to the closures of some (strongly) $n$-quasi-isotopic string links%
\footnote{I am indebted to M. Il'insky for pointing out a gap in what I thought was a proof of this assertion.}
(compare \cite{M24-1'}*{Example \ref{hash:nonlocal-knot}}).
But it is easy to prove a slightly weaker assertion.

We say two links in $S^3$ are {\it (strongly) $n$-quasi-isotopic via string links} if they are equivalent with
respect to the equivalence relation generated by ambient isotopy and the relation ``to be closures of (strongly) 
$n$-quasi-isotopic string links''. 

\begin{proposition} \label{string-quasi} (a) If two links in $S^3$ are $n$-quasi-isotopic, $n\ge 1$, then they 
are $(n-1)$-quasi-isotopic via string links.

(b) If two links in $S^3$ are strongly $n$-quasi-isotopic, then they 
are strongly $n$-quasi-isotopic via string links.
\end{proposition}

\begin{proof}[Proof. (a)] Let $L,L'\:mS^1\to S^3$ be the given links, related by an $n$-quasi-isotopy.
We may assume that the $n$-quasi-isotopy has only one double point.
Let $f\:mS^1\to S^3$ be the corresponding $n$-quasi-embedding.
It suffices to show that $f$ is ambient isotopic to the closure of an $(n-1)$-quasi-embedding $mI\to I^3$
of the string link boundary pattern.
Let $P_0\subset\dots\subset P_n\subset S^3$ and $J_0\subset\dots\subset J_n\subset mS^1$ be given by 
the definition an $n$-quasi-embedding.
Since the inclusion $P_{n-1}\to P_n$ is null-homotopic, it induces zero homomorphisms on the reduced 
cohomology groups.
Hence by the Alexander duality the inclusion $S^3\but P_n\to S^3\but P_{n-1}$ induces zero homomorphisms
on the reduced homology groups.
This implies in particular that any two points in $S^3\but P_n$ are connected by an arc in $S^3\but P_{n-1}$.
Let us recall that $f^{-1}(P_n)$ lies in the arc $J_n$.
Hence each component of $f(mS^1)$ contains some point which lies in $S^3\but P_n$.
Let $f(x_i)$ be such a point in the $i$th component of $f(mS^1)$.
Since $S^3\but P_{n-1}$ is a $3$-manifold, it contains a tree $T$ with leaves at the points $f(x_i)$.
Since $f(mS^1)$ is a $1$-manifold, $T$ may be assumed to meet it only in the leaves.
Let $B$ be a regular neighborhood of $T$, disjoint from $P_{n-1}\cup f(J_{n-1})$ and meeting each component 
of $f(mS^1)$ in an arc.
Then $Q:=\overline{S^3\but B}$ is a ball containing $P_{n-1}\cup f(J_{n-1})$ and such that $f^{-1}(Q)$
meets each component of  $mS^1$ in an arc.
Hence $f$ is ambient isotopic to the closure of an $(n-1)$-quasi-embedding $mI\to I^3$
of the string link boundary pattern.
\end{proof}

\begin{proof}[(b)] This is similar to the proof of (a), if the following observation is used.
Given a strong $n$-quasi-embedding $f\:mS^1\to S^3$ and the balls $B_1\subset\dots\subset B_n\subset S^3$
given by the definition a strong $n$-quasi-embedding, $S^3\but B_n$ is connected (since it is an open ball).
\end{proof}

\begin{remark}
As observed by J. Levine \cite{Le0}*{Proposition 6 (stated slightly differently)}, given a link $L$ in $S^3$, 
the choice of a string link whose closure is ambient isotopic to $L$ is equivalent to the choice of a PL 
$2$-disk $D$ in $S^3$ which meets every component $K_i$ of $L$ transversely in a single point $p_i$.
Let us note that $D$ deformation retracts onto a tree $T$ with leaves $p_1,\dots,p_n$ and with one vertex of valency $m$.
\end{remark}

\begin{lemma} \label{LA-lemma1}
Let $A$ and $B$ be $m$-component links in $S^3$ or string links.
Suppose that $K_A=K_B$ or more generally $K_A$ and $K_B$ are not separated by locally additive type $n$ invariants.
Then $A$ and $B$ are not separated by type $n$ invariants that are well-defined up to PL isotopy
if and only if $A$ and $B$ are not separated by locally additive type $n$ invariants.
\end{lemma}

\begin{proof} Similarly to the proof of Lemma \ref{LA-lemma0}.
\end{proof}

\begin{lemma} \label{LA-lemma2}
Let $A$ and $B$ be $m$-component links in $S^3$, and let $A'=A\#K_B$ and $B'=B\#K_A$
(which are PL isotopic to $A$ and $B$ respectively).
Then $A$ and $B$ are not separated by type $n$ invariants of string links that are well-defined up to PL isotopy
if and only if $A'$ and $B'$ are not separated by locally additive type $n$ invariants of string links.
\end{lemma}

\begin{proof} {\it ``If''.} Since every invariant of PL isotopy is locally additive, the hypothesis implies that
$A'$ and $B'$ are not separated by type $n$ invariants of string links that are well-defined up to PL isotopy.
On the other hand, let $\alpha$ and $\beta$ be some string links whose closures are $A$ and $B$, respectively.
Then $\alpha$ is PL isotopic to $\alpha':=\alpha\# K_\beta$, whose closure is $A'$.
Hence $A$ and $A'$ are not separated by type $n$ invariants of string links that are well-defined up to PL isotopy.
Similarly for $B'$ and $B$.

{\it ``Only if''.}
We are given string links $L_1,\dots,L_{2k}$ such that each $L_{2i-1}$ is not separated from $L_{2i}$
by type $n$ invariants, well-defined up to PL isotopy; the closures of $L_{2i}$ and $L_{2i+1}$ are ambient isotopic whenever both are defined; 
and the closures of $L_1$, $L_{2k}$ are ambient isotopic to $A$ and $B$, respectively.
Let us note that $K_{L_{2i}}=K_{L_{2i+1}}$ since these two string links are totally split and their closures are ambient isotopic.
Let $\Lambda_1=K_{L_{2k}}$ and $\Lambda_{2k}=K_{L_1}$.
Then the closures of $L_1\#\Lambda_1$ and $L_{2k}\#\Lambda_{2k}$ are ambient isotopic to $A'$ and $B'$, respectively.
For $i=1,\dots,k-1$ let $\Lambda_{2i}$ be a totally split string link representing $[K_{L_{2i}}]^{-1}[K_{L_1}\#K_{L_{2k}}]$
in the direct product of $m$ copies of Gusarov's group of $C_{n+1}$-equivalence classes of knots \cite{Gu1} (see also \cite{Hab}),
and let $\Lambda_{2i+1}=\Lambda_{2i}$.
Then the closures of $L_{2i}\#\Lambda_{2i}$ and $L_{2i+1}\#\Lambda_{2i+1}$ are ambient isotopic whenever both are defined,
and $K_{L_{2i}}\#\Lambda_{2i}$ and $K_{L_{2i+1}}\#\Lambda_{2i+1}$ are ambient isotopic themselves.
Also $K_{L_i}\#\Lambda_i$ is $C_{n+1}$-equivalent to $K_{L_1}\#K_{L_{2k}}$, and hence (see Proposition \ref{cn-eq})
is not separated from it by type $n$ invariants.
Then by Lemma \ref{LA-lemma1} each $L_{2i-1}\#\Lambda_{2i-1}$ is not separated from $L_{2i}\#\Lambda_{2i}$ by locally additive
type $n$ invariants.
Writing $L'_i=L_i\#\Lambda_i$, the sequence of string links $L'_1,\dots,L'_{2k}$ shows that $A'$ and $B'$ are not separated 
by locally additive type $n$ invariants of string links.
\end{proof}

\begin{proof}[Proof of Theorem \ref{rolfsen}] 
The previous results can be seen to yield the implications (1)--(9) shown in Figure \ref{implications} for every links $A$, $B$ in $S^3$,
where $A'=A\#K_B$ and $B'=B\#K_A$.

In more detail, the implication (1) holds by Theorem \ref{isotopy}; (2) is obvious; (3) holds by Proposition \ref{string-quasi};
(4) holds by Theorem \ref{2.2} (the case of string links); and (5) is obvious.
The equivalence (6) holds by Lemma \ref{LA-lemma2}; and (7) follows from Theorem \ref{locally-additive2}.
The implication (8) holds by Lemma \ref{fti-closure}; and (9) is obvious.
\end{proof}

\begin{figure}[h]
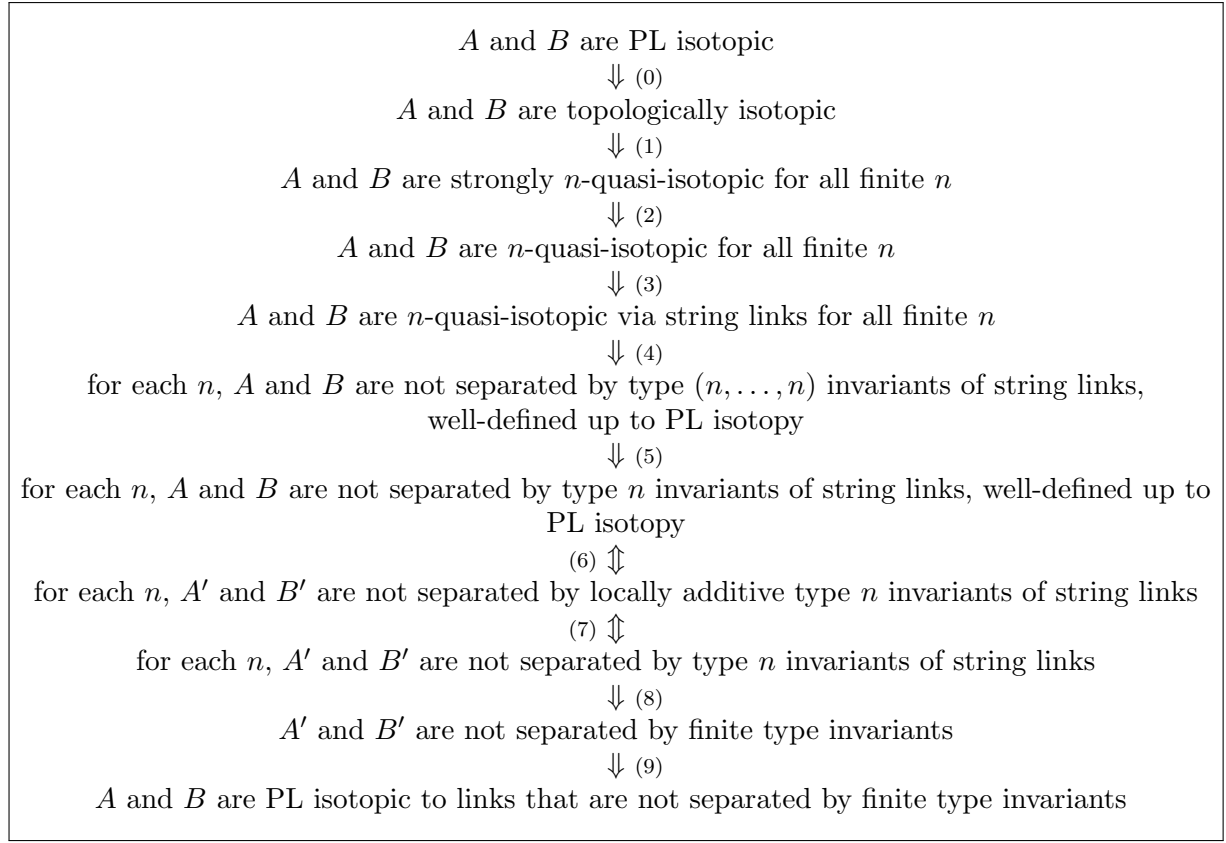

\noindent\fbox{\parbox{\textwidth}{\small\rm
\begin{center}
\medskip

$A$ and $B$ are PL isotopic

\phantom{\tiny (0)} $\Downarrow$ {\tiny (0)}

$A$ and $B$ are topologically isotopic

\phantom{\tiny (1)} $\Downarrow$ {\tiny (1)}

$A$ and $B$ are strongly $n$-quasi-isotopic for all finite $n$

\phantom{\tiny (2)} $\Downarrow$ {\tiny (2)}

$A$ and $B$ are $n$-quasi-isotopic for all finite $n$

\phantom{\tiny (3)} $\Downarrow$ {\tiny (3)}

$A$ and $B$ are $n$-quasi-isotopic via string links for all finite $n$

\phantom{\tiny (4)} $\Downarrow$ {\tiny (4)}

for each $n$, $A$ and $B$ are not separated by type $(n,\dots,n)$ invariants of string links, well-defined up to PL isotopy

\phantom{\tiny (5)} $\Downarrow$ {\tiny (5)}

for each $n$, $A$ and $B$ are not separated by type $n$ invariants of string links, well-defined up to PL isotopy

{\tiny (6)} $\Updownarrow$ \phantom{\tiny (6)}

for each $n$, $A'$ and $B'$ are not separated by locally additive type $n$ invariants of string links

{\tiny (7)} $\Updownarrow$ \phantom{\tiny (7)}

for each $n$, $A'$ and $B'$ are not separated by type $n$ invariants of string links

\phantom{\tiny (8)} $\Downarrow$ {\tiny (8)}

$A'$ and $B'$ are not separated by finite type invariants

\phantom{\tiny (9)} $\Downarrow$ {\tiny (9)}

$A$ and $B$ are PL isotopic to links that are not separated by finite type invariants
\medskip

\end{center}
}}
\smallskip

\caption{Some implications, where $A'=A\#K_B$ and $B'=B\#K_A$}
\label{implications}
\end{figure}

\section{Proof of Theorem \ref{main2}} \label{proofs2}

\begin{theorem}[Rolfsen \cite{Ro0}] \label{rolfsen-lemma0} If $L$ and $L'$ are links in $S^3$ which are PL isotopic,
and $K_L=K_{L'}$ up to an ambient isotopy, then $L$ and $L'$ are ambient isotopic.
\end{theorem}

Rolfsen's proof is based on the prime factorization of knots and some additional geometric constructions.
An alternative proof of Theorem \ref{rolfsen-lemma0} is given in \cite{M24-1'}; namely, it is noted there that 
Theorem \ref{rolfsen-lemma0} is an easy consequence of the prime factorization of links, whose new short proof 
is also included in \cite{M24-1'}.
Let us note that the assertion of Theorem \ref{rolfsen-lemma0} fails for links in $S^1\x S^2$ \cite{Ro1}*{Example 2}.

\begin{lemma} \label{destabilization}
Let $A$ and $B$ be $m$-component links in $S^3$ and $K$ be a totally split $m$-component link in $S^3$.
Then $A$ and $B$ are not separated by locally additive type $n$ invariants of string links if and only if $A\# K$ and $B\# K$ are not separated 
by locally additive type $n$ invariants of string links.
\end{lemma}

\begin{proof} {\it ``Only if''.}
We are given string links $L_1,\dots,L_{2k}$ such that the closures of $L_1$, $L_{2k}$ are ambient isotopic to $A$ and $B$, respectively;
the closures of $L_{2i}$ and $L_{2i+1}$ are ambient isotopic whenever both are defined; and each $L_{2i-1}$ is not separated from $L_{2i}$ 
by locally additive type $n$ invariants.
Let $\Lambda$ be a totally split string link whose closure is ambient isotopic to $K$, and let $L_i'=L_i\#\Lambda$.
Then the sequence of string links $L_1',\dots,L_{2k}'$ shows that $A\#K$ and $B\#K$ are not separated by locally additive type $n$ 
invariants of string links (it is here that we need the local additivity).

{\it ``If''.} Let $K'$ be a totally split string link representing $[K]^{-1}$ in the direct product of $m$ copies of Gusarov's group 
of $C_{n+1}$-equivalence classes of string knots \cite{Gu1} (see also \cite{Hab}).
Then by the ``only if'' part $A\#K\#K'$ and $B\#K\#K'$ are not separated by locally additive type $n$ invariants of string links.
Let $L$ be any string link whose closure is ambient isotopic to $A$.
Let $\Lambda$ and $\Lambda'$ be totally split string links whose closures are ambient isotopic to $K$ and $K'$, respectively.
Then the closure of $L\#\Lambda\#\Lambda'$ is ambient isotopic to $A\#K\#K'$.
On the other hand, $\Lambda\#\Lambda'$ is $C_{n+1}$-eqivalent to the string unlink.
Hence $L\#\Lambda\#\Lambda'$ is $C_{n+1}$-equivalent to $L$, and consequently (see Proposition \ref{cn-eq})
they are not separated by type $n$ invariants.
Thus $A$ and $A\#K\#K'$ are not separated by type $n$ invariants of string links.
Similarly, $B\#K\#K'$ and $B$ are not separated by type $n$ invariants of string links.
\end{proof}

\begin{proof}[Proof of Theorem \ref{main2}] 
The implications (1)--(9) in Figure \ref{implications} were established in the proof of Theorem \ref{rolfsen}, whereas 
the implication (0) is obvious.

Let (X) denote the converse to the composite of the implications (0)--(9) (i.e.\ {\it ``if $A$ and $B$ are PL isotopic to links
which are not separated by finite type invariants, then $A$ and $B$ are PL isotopic''}).
Then (X) clearly follows from \ref{habiro}, and by Theorem \ref{rolfsen-lemma0}, (X)$\land$\ref{vassiliev} implies \ref{habiro}.

Let (Y) denote the converse to the composite of the implications (8) and (9) (i.e.\ {\it ``if $A$ and $B$ are PL isotopic to links that are not 
separated by finite type invariants, then for each $n$, $A'$ and $B'$ are not separated by type $n$ invariants of string links''}).

Let us show that (Y)$\land$\ref{vassiliev} implies \ref{habegger-meilhan}.
Suppose that $A$ and $B$ are not separated by finite type invariants.
Then by \ref{vassiliev} $K_A=K_B$, and by (Y) for each $n$, $A'$ and $B'$ are not separated by type $n$ invariants of string links.
We have $A'=A\#K$ and $B'=B\#K$, where $K=K_A=K_B$, so by Lemma \ref{destabilization} for each $n$, $A$ and $B$ are not separated by 
locally additive type $n$ invariants of string links.
Then by Theorem \ref{locally-additive2} for each $n$, $A$ and $B$ are not separated by type $n$ invariants of string links.

Let us show that \ref{habegger-meilhan}$\land$\ref{vassiliev} implies (Y).
Suppose that $A$, $B$ are PL isotopic to links $\bar A$, $\bar B$ (respectively), which are not separated by finite type invariants.
Then by \ref{vassiliev} $K_{\bar A}=K_{\bar B}$, and by \ref{habegger-meilhan} for each $n$, 
$\bar A$ and $\bar B$ are not separated by type $n$ invariants of string links.
The link $A'$ is PL isotopic to $A$ and hence to $\bar A$, and similarly $B'$ is PL isotopic to $\bar B$.
Therefore $\bar A\#Q$ is ambient isotopic to $A'\#Q'$ and $\bar B\#R$ is ambient isotopic to $B'\#R'$
for some totally split links $Q$, $Q'$, $R$ and $R'$.
Then $\bar A\#Q\#R$ is ambient isotopic to $A'\#Q'\#R$ and $\bar B\#Q\#R$ is ambient isotopic to $B'\#Q\#R'$.
Since $K_{\bar A}=K_{\bar B}$, we have $K_{\bar A}\#Q\#R=K_{\bar B}\#Q\#R$, and therefore also
$K_{A'}\#Q'\#R=K_{B'}\#Q\#R'$.
But on the other hand, we have $K_{A'}=K_A\#K_B=K_{B'}$.
Hence by the uniqueness of factorization into prime knots, $Q'\#R=Q\#R'$.
Thus, writing $S=Q\#R$ and $S'=Q'\#R=Q\#R'$, we get that $\bar A\# S=A'\#S'$ and $\bar B\#S=B'\#S'$.
Since for each $n$, $\bar A$ and $\bar B$ are not separated by type $n$ invariants of string links, 
it follows by a double application of Lemma \ref{destabilization} that for each $n$, 
$A'$ and $B'$ are not separated by additive type $n$ invariants of string links.
Then by Theorem \ref{locally-additive2} for each $n$, $A'$ and $B'$ are not separated by type $n$ invariants of string links.

Let (Z) denote the converse to the composite of the implications (0)--(5) (i.e.\ {\it ``if $A$ and $B$ can be represented for each $n$ 
as closures of string links that are not separated by type $n$ invariants, well-defined up to PL isotopy, then $A$ and $B$ are 
PL isotopic''}).
Then we get that 
\ref{habiro}$\Leftrightarrow$\ref{vassiliev}$\land$(X)$\Leftrightarrow$\ref{vassiliev}$\land$(Y)$\land$(Z)$\Leftrightarrow$
\ref{vassiliev}$\land$\ref{habegger-meilhan}$\land$(Z).
\end{proof}

\section{Proof of Theorem \ref{string+rational}} \label{proofs3}

The following string link version of Theorem \ref{rolfsen-lemma0} is proved in \cite{M24-1'}.

\begin{theorem} \cite{M24-1'} \label{rolfsen-lemma0'} If $L$ and $L'$ are string links which are PL isotopic,
and $K_L=K_{L'}$ up to ambient isotopy, then $L$ and $L'$ are ambient isotopic.
\end{theorem}

The following more precise version of Lemma \ref{LA-lemma0} is proved by the same argument.

\begin{lemma} \label{LA-lemma3}
Let $A$ and $B$ be $m$-component links in $S^3$ or string links.
Then $A$ and $B$ are not separated by type $n$ invariants that are well-defined up to PL isotopy
if and only if $A':=A\#K_B$ and $B':=B\#K_A$ are not separated by locally additive type $n$ invariants.
\end{lemma}

Let us reformulate part (b) of Theorem \ref{string+rational} in more detail:

\begin{theorem}\label{string} 
Finite type invariants separate string links if and only if finite type invariants separate string knots (or equivalently knots in $S^3$)
and finite type invariants, well-defined up to PL isotopy, separate PL isotopy classes of string links.
\end{theorem}

\begin{proof}
By Lemma \ref{LA-lemma3} and Theorem \ref{locally-additive2}, the string link version of \ref{habiro} implies 
the string link version of \ref{links-modulo-knots}.

By Theorem \ref{rolfsen-lemma0'} the conjunction of the string link versions of \ref{vassiliev} and \ref{links-modulo-knots} implies
the string link version of \ref{habiro}.
\end{proof}

\begin{remark} \label{kontsevich}
In order to prepare for the proof of part (a) of Theorem \ref{string+rational} let us review some basics of 
the Kontsevich integral.
Since much of the literature on the Kontsevich integral focuses on the case of knots, we focus on references 
which cover the case of links (in the present remark, all knots and links are in $S^3$, and all tangles are in $I^3$).
\begin{itemize} 
\item A detailed exposition of M. Kontsevich's original (analytical) definition for knots appears in \cite{CDM}.
A generalization of the Kontsevich integral to tangles, along with a proof of its invariance for tangles, 
is sketched in \cite{LM1}*{\S1}; when guided by this sketch, some details of the proof can be found in \cite{BN} 
and, when guided by \cite{BN}, some further details can be found in \cite{CDM} and \cite{ChD}.
\item There is also a combinatorial version of the Kontsevich integral, first described by 
P. Cartier, Le--Murakami and S. Piunikhin (independently); its construction for ``non-associative tangles'', which include links, 
along with a proof of invariance, can be found in \cite{BN2} (concerning the definition see also \cite{CDM}).
\item That the two versions of the Kontsevich integral are equal for links is proved in \cite{LM1} (see also \cite{CDM}) 
and by a different method in \cite{AF}.
\item It is shown in \cite{LM2} (see also \cite{CDM}) that the combinatorial Kontsevich integral $Z$ of links assumes only 
rational values.
In more detail, let $V_m$ be the vector space of all rational formal linear combinations of chord diagrams on $m$ circles,
let $A_m$ be its quotient by the $4$-term and the $1$-term relations,%
\footnote{See \cite{CDM} concerning chord diagrams and the $4$-term and the $1$-term relations.}
which is graded by the number of chords: 
$A_m=\bigoplus_{n=0}^\infty A_{mn}$, and let $\hat A_m$ be the graded completion of $A_m$, that is, $\prod_{n=0}^\infty A_{mn}$.
Then, as explained in \cite{LM1}, \cite{LM2} and \cite{HaMa}, $Z(L)$ takes values in $\hat A_m$, where $m$ is the number
of components of $L$.
\item A proof sketch that the combinatorial Kontsevich integral is a universal rational%
\footnote{By a rational invariant we mean an invariant with values in a vector space over $\Q$.}
finite type invariant of links appears in \cite{BN2}.
As noted in \cite{LM2}*{\S5} and \cite{HaMa}*{Remark 3.3}, the proof in \cite{BN} that the analytic Kontsevich integral 
is a universal rational finite type invariant of knots (see also \cite{CDM} concerning this proof) caries over to 
the case of links.
\item The ``Fundamental Theorem'' that every weight system is the symbol of some finite type invariant 
is discussed in the case of links in \cite{CDM}*{Theorem in \S5.10.1 and the proof of Theorem 8.8.2}.
\item As noted in \cite{LM2}*{Theorem 5}, the combinatorial Kontsevich integral, when suitably normalized 
(see \cite{CDM}*{\S8.7.2} concerning the normalization), is multiplicative with respect to the connected sum 
of links along selected components (see \cite{M24-1'} concerning the latter) and the similarly defined 
connected sum of chord diagrams along selected components:
$Z(L\#_{i,j}L')=Z(L)\#_{i,j} Z(L')$.%
\footnote{In more detail, if $d$ is a chord diagram on $m$ circles and $i\in\{1,\dots,m\}$, then $d[i]$ denotes the chord 
diagram on $m-1$ circles and one arc, obtained by removing from the $i$th circle a small open arc disjoint from all chords.
This $d[i]$ is well-defined (i.e.\ does not depend on the choice of the small open arc) due to the $4$-term relation
(see \cite{CDM}*{proof of Lemma in \S4.4.3}).
Also, if $d$ and $d'$ are chord diagrams on oriented $1$-manifolds $\Theta$ and $\Theta'$ with $\partial\Theta=
\partial I=\partial\Theta'$, then $d\# d'$ denotes the chord diagram on the closed $1$-manifold 
$\Theta\cup_{\partial\Theta=\partial\Theta'}\Theta'$ consisting of all chords of $d$ and all chords of $d'$.
Both operators extend linearly over the graded completions of spaces of chord diagrams.
Finally, $d\#_{i,j}d'$ denotes $d[i]\#d'[j]$.}
The same formula for the analytical Kontsevich integral can be seen to follow from the multiplicativity of 
the preliminary Kontsevich integral under composition of tangles (see \cite{CDM}*{8.4.3}).
\item Since connected sum is single-valued for chord diagrams on one circle, $A_1$ is an algebra, and $\hat A_1$
is an algebra where every element of the form $1+$(terms of positive degrees) is invertible (cf.\ \cite{CDM}).
In particular, $Z(K)$ is invertible for every knot $K$.
\end{itemize}
\end{remark}

Let us reformulate part (a) of Theorem \ref{string+rational} in more detail:

\begin{theorem}\label{rational}
Rational finite type invariants separate links in $S^3$ if and only if rational finite type invariants separate knots in $S^3$ and 
rational finite type invariants, well-defined up to PL isotopy, separate PL isotopy classes of links in $S^3$.
\end{theorem}

\begin{proof}
Let $\bar Z(L)=\dfrac{Z(L)}{Z(K_1)\cdots Z(K_m)}$, where $K_1,\dots,K_m$ are the components of $L$,
and the fraction is understood more precisely as $\big(\cdots\big(Z(L)\#_{1,1}Z(K_1)^{-1}\big)\cdots\big)\#_{n,1}Z(K_1)^{-1}$.
Then $\bar Z(L)$ is invariant under PL isotopy; also its coefficients are easily seen to be finite type invariants.

Let us show that the rational version of \ref{habiro} implies the rational version of \ref{links-modulo-knots}.
Suppose that $A$ and $B$ are links in $S^3$ which are not separated by rational finite type invariants,
well-defined up to PL isotopy.
Then in particular $\bar Z(A)=\bar Z(B)$.
Let $A'=A\#K_B$ and $B'=B\#K_A$.
Then $A'$ is PL isotopic to $A$, so $\bar Z(A')=\bar Z(A)$; and similarly $\bar Z(B')=\bar Z(B)$.
On the other hand, $K_{A'}$ is ambient isotopic to $K_{B'}$, so the denominators of the equal fractions 
$\bar Z(A')$ and $\bar Z(B')$ are equal.
Hence so are their numerators $Z(A')$ and $Z(B')$.
Now by the universality of $Z$, the rational version of \ref{habiro} implies that $A'$ is ambient isotopic to $B'$.
Hence $A$ is PL isotopic to $B$.

By Theorem \ref{rolfsen-lemma0} the conjunction of the rational versions of \ref{vassiliev} and \ref{links-modulo-knots} implies 
the rational version of \ref{habiro}.
\end{proof}

\section{Reduced Conway polynomial} \label{conway}

The {\it Conway polynomial} $\nabla_L$ of a link $L$ in $S^3$ is a normalized sign-refined version of the Alexander polynomial;
it can be defined for instance in terms of Seifert matrices (see \cite{Lic}).
The Conway polynomial is characterized (see \cite{Lic}*{Theorem 8.6}) by two axioms:
\begin{gather*}\nabla_{\text{unknot}}=1,\\
\nabla_{L_+}(z)-\nabla_{L_-}(z)=z\nabla_{L_0}(z),\tag{C}\label{conway-skein}
\end{gather*}
where $L_0$, $L_+$ and $L_-$ agree outside a small ball, and inside this ball they are as follows:%
\footnote{As explained in \S\ref{fti}, $L_+$ and $L_-$ can be distinguished from each other without using the plane diagram.}
\smallskip

\begin{center}
\includegraphics[width=0.4\linewidth]{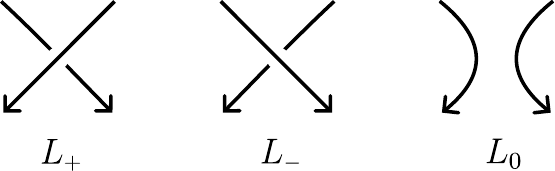}
\end{center}
\smallskip

Implicit in this figure is the singular link $L$ (see the figure related to formula (\ref{vassiliev-extension})).
Let us note that $L_0$ has one more component than $L_+$ and $L_-$ if the intersection in $L$ is a self-intersection of some component, 
and one less component if the intersection involves distinct components.
We will refer to $L_+$ and $L_-$ as the positive and negative {\it resolutions} of $L$, and $L_0$ as the {\it smoothing} of $L$.

Suppose that we are given a class $\CC$ of links whose Conway polynomials we know, and a link $L$ whose Conway polynomial we want
to compute.
Let us define a {\it computation tree} for $L$ with respect to $\CC$.
A computation tree $T_0$ of order $0$ is a homotopy $h_t$ from $L$ to some $L'\in\CC$ through links and singular links with one double point.
The singular links of $h_t$ are called the {\it buds} of $T_0$.
A computation tree $T_n$ of order $n\ge 1$ consists of a computation tree $T_{n-1}$ of order $n-1$, and, for each bud $\Lambda$ of $T_{n-1}$, of 
a homotopy $h_t^\Lambda$ from the smoothing $\Lambda_0$ of $\Lambda$ to some $\Lambda'\in\CC$ through links and singular links with one double point.
The singular links of each $h_t^\Lambda$ are called the {\it buds} of $T_n$.
We also call $L$ the {\it root}, the homotopies added when forming $T_k$ the {\it order $k$ branches}, and their final links, which are
in $\CC$, the {\it order $k$ leaves} of $T_n$.
 
Since every link is homotopic to a trivial link, for every link $L$ there exists a computation tree of an arbitrarily high order, 
whose leaves are trivial links. 
If we want the computation to be guaranteed to terminate in finitely many steps, a slightly more elaborate construction is needed 
(see Remark \ref{no-buds} below).
 
\begin{lemma} {\rm (cf.\ \cite{Lic}*{Proposition 8.7})} \label{lickorish}
(a) $\nabla_L(z)=0$ for any split link $L$.

(b) The Conway polynomial of an $m$-component link $L$ is of the form 
\[\nabla_L(z)=z^{m-1}\big(c_0+c_1z^2+c_2z^4+\dots+c_rz^{2r}\big).\]

(c) $c_0(K)=1$ for a knot $K$, and $c_0(L)=\lk(L)$ for a $2$-component link $L$.
\end{lemma}

\begin{proof}[Proof (sketch). (a)] Represent $L$ as $\Lambda_0$, where $\Lambda_+$ and $\Lambda_-$ are ambient isotopic.
\end{proof}

\begin{proof}[(b)] We need to show that the coefficient of $\nabla_L$ at $z^n$ is zero if $n<m-1$ or $n$ has the same parity as $m$.
If $n<m-1$, there is a computation tree for $L$ of order $n$ whose all leaves are split links.
In general there is a computation tree for $L$ of order $n$ whose leaves are split links and unknots; each leaf of an order $\equiv m\pmod 2$
has an even number of components, and so cannot be an unknot.
\end{proof}

\begin{proof}[(c)] Use any computation tree for $L$ of order $0$, respectively $1$, whose leaves are split links and unknots.
\end{proof}

\begin{remark} \label{hhh}
It is easy to see that for $3$-component links $c_0(L)=ab+bc+ca$, where $a,b,c$ are the linking numbers
of the $2$-component sublinks (cf.\ \cite{Liv}).
For any number of components, $c_0(L)$ is a symmetric polynomial in the pairwise linking numbers $l_{ij}$.
Namely, $c_0(L)=\det(\Lambda^{(p)})$ for any $p\in\{1,\dots,m\}$, where $\Lambda=(\lambda_{ij})$ is 
the $m\x m$ matrix with entries
\[\lambda_{ij}=\begin{cases}-l_{ij},&\text{if }i\ne j\\
\sum_{k\ne i}l_{ik}&\text{if }i=j
\end{cases}\] and $\Lambda^{(p)}$ denotes the $(m-1)\x(m-1)$ matrix obtained by removing from $\Lambda$
the $p$th row and the $p$th column (Hosokawa, Hartley and Hoste; see \cite{Le}*{Proposition 3.2} for a proof 
and \cite{Masb} for the references).
Hartley and Hoste also reformulated the same expression as a sum over all spanning trees $T$ of the 
complete graph $K_m$: \[c_0(L)=\sum_T\prod_{\{i,j\}\in E(T)}l_{ij}.\]
\end{remark}

\begin{lemma} \label{conway-coefficients}
For an $m$-component link $L$ each $c_n(L)$ is of type $m-1+2n$ and of colored type $2n$.
\end{lemma}

The first assertion is well-known \cite{BN}*{Theorem 2}, and the second is known in a slightly different form
\cite{ShYa}*{Lemma 3.5}.

\begin{proof} The skein relation (\ref{conway-skein}) is reminiscent of the formula (\ref{vassiliev-extension}) in \S\ref{fti} 
describing the extension of a link invariant to singular links.
The two formulas together imply the relation 
$\nabla^\x_{\includegraphics[width=0.4cm]{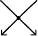}}=z\nabla^\x_{\includegraphics[width=0.4cm]{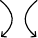}}$, 
where $\nabla^\x$ is the standard extension of $\nabla$ to singular links.
It follows that $\nabla^\x_\Lambda$ is divisible by $z^{n+1}$ for any singular link $\Lambda$ with $n+1$ double points.
Therefore the coefficient of $\nabla_L$ at $z^n$ is a type $n$ invariant.

Now suppose that $\Lambda$ is a singular link with $m$ components and $n+1$ double points, whose all double points are 
self-intersections of components.
Let us color the components of $\Lambda$ in distinct colors.
Then after smoothing each double point we obtain a non-singular link $\Lambda_{0\dots0}$ which is still colored in $m$ colors 
(each component being colored only in one color, and each of the $m$ colors being used).
This means that $\Lambda_{0\dots0}$ has at least $m$ components.
Then by Lemma \ref{lickorish}(b) $\nabla_{\Lambda_{0\dots0}}$ is divisible by $z^{m-1}$.
Hence $\nabla^\x_\Lambda$ is divisible by $z^{(m-1)+(n+1)}$.
Since $\Lambda$ was an arbitrary singular link with $n+1$ double points that are all self-intersections of components,
the coefficient of $\nabla_L$ at $z^{m-1+n}$ is a colored type $n$ invariant.
\end{proof}

\begin{remark} \label{no-buds}
For every link $L$ there exists a finite computation tree for $L$ which has no buds, and whose leaves are trivial links.
Namely, given a plane diagram $D$ of $L$, this $D$ is also a plane diagram of a trivial link $L'$ (see \cite{PS}*{Theorem 3.8}).%
\footnote{Thinking of $D$ as lying in the $xy$ plane, each component of $L'$ consists of arc which projects homeomorphically 
onto the $z$ axis and a straight line segment whose projection onto the $xy$-plane is disjoint from the crossings of $D$ 
(cf.\ \cite{PS}*{Figure 3.10}). 
The projections of the components on the $z$ axis are pairwise disjoint.}
In order to get from $L$ to $L'$, we only need to switch some of the crossings in $D$.
For each singular link $\Lambda$ in this homotopy, its smoothing $\Lambda_0$ has a plane diagram which has one less 
crossing than $D$.
Then we repeat the same construction for this plane diagram, and so on.
This process terminates after finitely many steps, since the number of crossings in the plane diagrams being considered 
decreases in each step.
One consequence of this computation is that the degree of $\nabla_L$ is bounded from above by the number of crossings in $D$.
\end{remark}

\begin{lemma} \label{insertion} {\rm (cf.\ \cite{Lic}*{Proposition 16.2(ii)})} 
$\nabla_{L\#_{i,j} L'}=\nabla_L\nabla_{L'}$.
\end{lemma}

Here $L\#_{i,j} L'$ is the connected sum of $L$ and $L'$ along selected components (see \cite{M24-1'}).

\begin{proof}[Proof (sketch)]
The computation tree for $L$ described in Remark \ref{no-buds} yields that
$\nabla_L=\sum_{k=0}^n z^{2k}\sum_{i=1}^{r_k}\epsilon_{ki}\nabla_U$, where $U$ is the unknot and $\epsilon_{ki}=\pm1$.
If the same computation tree is applied to a plane diagram of $L\#_{i,j} L'$ obtained from the diagram $D$ of $L$ by adjoining 
a diagram of $L'$ lying in a small disk, so that the contents of the small disk is kept intact, then it yields
$\nabla_{L\#_{i,j} L'}=\sum_{k=0}^n z^{2k}\sum_{i=1}^{r_k}\epsilon_{ki}\nabla_{L'}$, where $r_k$ and $\epsilon_{ki}$ are same as before.
Since $\nabla_U=1$, the two expressions combine to yield the desired result.
\end{proof}

Since $c_0(K)=1$ for every knot $K$, we have the formal power series
\[\bar\nabla_L(z):=\dfrac{\nabla_L(z)}{\nabla_{K_1}(z)\cdots\nabla_{K_m}(z)},\] 
where $K_1,\dots,K_m$ are the components of $L$.

From Lemma \ref{insertion} we obtain

\begin{corollary} \cite{Tr2}, \cite{Ro4} \label{locknot} $\bar\nabla_L$ is invariant under PL isotopy.
\end{corollary}

\begin{proposition} \label{reduced-coefficients}
If $L$ is an $m$-component link, then $\bar\nabla_L$ is of the form 
\[\bar\nabla_L(z)=z^{m-1}\big(\alpha_0+\alpha_1z^2+\alpha_2z^4+\dots\big),\]
with each $\alpha_i(L)=c_i(L)-\big(\alpha_{i-1}(L)c_1(K)+\dots+\alpha_0(L)c_i(K)\big)$, where
$K=K_1\#\dots\#K_m$.
\end{proposition}

\begin{proof} By Lemma \ref{insertion} $\nabla_{K_1}\cdots\nabla_{K_m}=\nabla_K$.
Then the definition of $\bar\nabla_L$ implies 
$\bar\nabla_L=\nabla_L-\bar\nabla_L(\nabla_K-1)$.
The assertion follows from this formula and Lemma \ref{lickorish}(b).
\end{proof}

\begin{corollary} If $L$ is a $2$-component link, then $\alpha_0(L)=c_0(L)$ is its linking number and 
$\alpha_1(L)=c_1(L)-c_0(L)\big(c_1(K_1)+c_1(K_2)\big)$ is its generalized Sato--Levine invariant.
\end{corollary}

See Example \ref{conway-ex} concerning the generalized Sato--Levine invariant.
The following crossing change formula for the generalized Sato--Levine invariant is well-known.

\begin{proposition} \label{beta-jump} {\rm \cite{Liv} (see also \cite{Tr2}*{Theorem 10.2}, \cite{KL}, \cite{AR})}
Let $L=(K_1,K_2)$ be a $2$-component link. 
Then the jump of $\alpha_1(L)$ under a self-intersection of $K_1$ is 
\[\alpha_1(L_+)-\alpha_1(L_-)=l_{\eta2}l_{\zeta2},\]
where $l_{ij}=\lk(K_i,K_j)$ and $K_\eta$, $K_\zeta$ are the components of the smoothing $K_0$ of the singular knot.%
\footnote{We may think of $\eta$ and $\zeta$ as the two primitive $3$rd roots of $1$.}
\end{proposition}

\begin{proof} We have 
\[\alpha_1(L_+)-\alpha_1(L_-)=\big(c_1(L_+)-c_1(L_-)\big)-c_0(L_+)c_0(K_0)=c_0(L_0)-c_0(L_+)c_0(K_0),\]
using that $c_0(L_+)=c_0(L_-)$.
Now \[c_0(L_0)=l_{\eta2}l_{\zeta2}+l_{\eta2}l_{\eta\zeta}+l_{\zeta2}l_{\eta\zeta}.\]
The latter two summands are canceled by $c_0(L_+)c_0(K_0)=l_{12}l_{\eta\zeta}$.
The remaining summand is the one in the statement.
\end{proof}

In order to similarly understand $\alpha_1(L)$ for a $3$-component link $L$, it is convenient to introduce a correction term.
Given a $3$-component link $L$, let us consider \[\gamma(L):=\alpha_1(L)-\sum_{(\Lambda,\Lambda')}\alpha_1(\Lambda)\alpha_0(\Lambda'),\]
where $(\Lambda,\Lambda')$ runs over all ordered pairs of distinct $2$-component sublinks of $L$.

\begin{proposition} \label{gamma} Let $L=(K_1,K_2,K_3)$ be a $3$-component link. 
Then the jump of $\gamma(L)$ under a self-intersection of $K_1$ is 
\[\gamma(L_+)-\gamma(L_-)=l_{23}(l_{\eta2}l_{\zeta3}+l_{\zeta2}l_{\eta3}),\]
where $l_{ij}=\lk(K_i,K_j)$ and $K_\eta$, $K_\zeta$ are the components of the smoothing $K_0$ of the singular knot.
\end{proposition}

\begin{proof} We have $\alpha_1(L)=c_1(L)-c_0(L)\big(c_1(K_2)+c_1(K_2)+c_1(K_3)\big)$.
Then 
\[\alpha_1(L_+)-\alpha_1(L_-)=\big(c_1(L_+)-c_1(L_-)\big)-c_0(L_+)c_0(K_0)=c_0(L_0)-c_0(L_+)c_0(K_0),\]
using that $c_0(L_+)=c_0(L_-)$.
The formula of Remark \ref{hhh} expresses $c_0(L_0)$ as a sum of 16 summands (corresponding to the 15 spanning trees of $K_4$).
It can be checked that 8 of these summands cancel with $c_0(L_+)c_0(K_0)$ and another 6 cancel with the jump of
the correction term.
The remaining 2 summands are the ones in the statement.
\end{proof} 

\begin{remark} It is observed in \cite{Msol}*{Corollary \ref{sol:3-comp-bar}} that $\gamma(L)$ is, up a certain polynomial in the pairwise
linking numbers, the coefficient at $z_1z_2z_3$ of the power series obtained by expanding the Conway potential function $\Omega_L(x_1,x_2,x_3)$
in Conway's variables $z_i=x_i-x_i^{-1}$.
Given this, Proposition \ref{gamma} becomes an easy consequence of Conway's first skein relation for $\Omega_L$
(cf.\ \cite{Msol}*{Proposition \ref{sol:jump}(a)}).
\end{remark}

\begin{corollary} \label{brunnianity}
(a) $\gamma(L)$ is not a function of any invariants of proper sublinks of $L$.

(b) $\lambda(L)\gamma(L)$ is not a function of any invariants of proper sublinks of $L$, where $\lambda(L)$
is the product of the pairwise linking numbers of $L$.
\end{corollary}

Part (a) can also be proved in a very different way, using that if the pairwise linking numbers vanish, then $c_3(L)=\mu(123)^2$ \cite{Co2}*{Theorem 5.1}.
A different proof of part (b) is given in \cite{Msol}*{Proposition \ref{sol:brunnianity}(b)}.

\begin{proof} Let $\mu_2$ be a $C_2$-move decomposed as follows:
\[\mu_2\:L_+\overset{\mu_1}{\rightsquigarrow} L_-\overset{\nu_2}{\rightsquigarrow} L'_-\overset{\mu_1^{-1}}{\rightsquigarrow} L'_+,\]
where $\mu_1$ is a positive self-intersection of the first component, $\mu_1^{-1}$ is the reverse motion of the first component, 
and $\nu_2$ is an ambient isotopy with support in the second component.
Let $K_3$ denote the third component of $L_\pm$, $L'_\pm$, let $K_2$ (resp.\ $K_{2'}$) denote the second component of $L_\pm$ (resp.\ $L'_\pm$),
and let $K_\eta$, $K_\zeta$ be the components of the smoothing of the singular knot of $\mu_1$ and $\mu_1^{-1}$.
Writing $l_{ij}=\lk(K_i,K_j)$, by Proposition \ref{gamma} \[\gamma(L'_+)-\gamma(L_+)=l_{23}(l_{\eta2'}l_{\zeta3}+l_{\zeta2'}l_{\eta3})-
l_{23}(l_{\eta2}l_{\zeta3}+l_{\zeta2}l_{\eta3})=l_{23}(l_{\zeta3}-l_{\eta3}),\]
where the last equality is up to a sign, but with an appropriate choice of $\nu_2$ the sign is positive.
Let $\mu_3$ be a $C_3$-move decomposed as follows:
\[\mu_3\:L_+\overset{\mu_2}{\rightsquigarrow} L'_+\overset{\nu_3}{\rightsquigarrow} \tilde L'_+\overset{\mu_2^{-1}}{\rightsquigarrow}\tilde L_+,\]
where $\mu_2^{-1}$ is the motion of the first two components reverse to $\mu_2$, and $\nu_3$ is an ambient isotopy with support in the third component.
Let $K_{\tilde 3}$ denote the third component of $\tilde L_\pm$, $\tilde L'_\pm$.
Then \[\gamma(\tilde L_+)-\gamma(L_+)=l_{23}(l_{\zeta\tilde 3}-l_{\eta\tilde 3})-l_{23}(l_{\zeta3}-l_{\eta3})=2l_{23},\]
where the last equality is up to a sign, but with an appropriate choice of $\nu_3$ the sign is positive.
On the other hand, it is easy to see that $\mu_3$ restricted to any proper sublink of $L_+$ can be effected by an ambient isotopy.
Hence if $v(L)$ is any function of invariants of proper sublinks of $L$, then $v(\tilde L_+)-v(L_+)=0$.
This proves (a), and a similar argument proves (b).
\end{proof}

\begin{proposition} \label{reduced-coefficients2}
For an $m$-component link $L$ each $\alpha_n(L)$ is 

(a) of type $m-1+2n$ and of colored type $2n$;

(b) of colored type $2n-1$ if $n>0$.
\end{proposition}

\begin{proof}[Proof. (a)] By Lemma \ref{insertion} $\nabla_K=\nabla_{K_1}(z)\cdots\nabla_{K_m}(z)$, where $K_1,\dots,K_m$ 
are the components of $L$ and $K=K_1\#\dots\#K_m$.
Hence $c_i(K)=\sum_{i_1+\dots+i_m=i}c_{i_1}(K_1)\cdots c_{i_m}(K_m)$.
By Lemma \ref{conway-coefficients} each $c_k(K_j)$ is a type $2k$ invariant of $K_j$, and hence 
a type $2k$ invariant of $L$ (cf.\ Lemma \ref{knot-link}(a)).
Then by Corollary \ref{product}(a) $c^*_i(L):=c_i(K)$ is a type $2i$ invariant of $L$.

By Proposition \ref{reduced-coefficients} each 
$\alpha_n(L)=c_n(L)-\big(\alpha_{n-1}(L)c^*_1(L)+\dots+\alpha_0(L)c^*_n(L)\big)$.
By Lemma \ref{conway-coefficients} $\alpha_0(L)=c_0(L)$ is of type $m-1$ and of colored type $0$.
Now it follows by induction, using Corollary \ref{product}(a), that $\alpha_n(L)$ is a type $m-1+2n$ invariant 
and a colored type $2n$ invariant of $L$.
\end{proof}

\begin{proof}[(b)] We will use the notation from the proof of (a).
We need to show that $\alpha_k(L)$ is of colored type $2k-1$ for each $k>0$.
Suppose that this is known for $i=1,\dots n-1$.
Then it follows from Corollary \ref{product}(a) that $\big(\alpha_{n-1}(L)c^*_1(L)+\dots+\alpha_1(L)c^*_{n-1}(L)\big)$
is of colored type $2n-1$.
Since $\alpha_n(L)=c_n(L)-\big(\alpha_{n-1}(L)c^*_1(L)+\dots+c_0(L)c^*_n(L)\big)$, it remains to show that 
$c_n(L)-c_0(L)c^*_n(L)$ is a colored type $2n-1$ invariant of $L$.

Let $\nabla^*_L=\nabla_K$. 
Thus we have $\nabla^*_L=c_0^*(L)+c_1^*(L)z^2+\dots+c_n^*(L)z^{2n}$ for some $n=n(L)$, where $c_0^*(L)=1$.
Then $c_n(L)-c_0(L)c^*_n(L)$ is the coefficient at $z^{m-1+2n}$ of the polynomial 
$\hat\nabla_L:=\nabla_L-z^{m-1}c_0(L)\nabla^*_L$.
More generally, for a link $L$ colored in $m$ colors so that no color is omitted, let 
$\nabla^*_{L}=\nabla_{L^1}(z)\cdots\nabla_{L^m}(z)$, where $L^i$ is the sublink of $L$ of the $i$th color,
and let $\hat\nabla_L=\nabla_L-z^{m-1}\lambda(L)\nabla^*_L$, where $\lambda(L)=\sum_T\prod_{\{i,j\}\in E(T)}\lk(L^i,L^j)$,
the sum being over all spanning trees $T$ of the complete graph $K_m$ (cf.\ Remark \ref{hhh}).

Let $\Lambda$ be a singular link with $m$ components and $2n$ double points, whose all double points are 
self-intersections of components.
The standard extension $\hat\nabla^\x$ of $\hat\nabla$ to singular links is easily seen to satisfy
$\hat\nabla^\x_\Lambda=\nabla^\x_\Lambda-z^{m-1}\lambda(\Lambda)\nabla^{*\x}_\Lambda$,
where $\lambda$ is extended by continuity to singular links which are link maps.
This implies the relation 
$\nabla^\x_{\includegraphics[width=0.4cm]{1x.pdf}}=z\nabla^\x_{\includegraphics[width=0.4cm]{1op.pdf}}$.
Let $\Lambda_{0\dots0}$ be the non-singular link obtained by smoothing all the double points of $\Lambda$.
As explained in the proof of Lemma \ref{conway-coefficients} $\Lambda_{0\dots0}$ is colored in $m$ colors
so that no color is omitted.

If $\Lambda_{0\dots0}$ has precisely $m$ components, then it follows from Remark \ref{hhh} that $\hat\nabla_{\Lambda_{0\dots0}}$ 
is divisible by $z^m$.
If $\Lambda_{0\dots0}$ has more than $m$ components, then by Lemma \ref{lickorish}(b) $\nabla_{\Lambda_{0\dots0}}$ 
is divisible by $z^m$ and also $\nabla_{\Lambda_{0\dots0}^i}$ is divisible by $z$ for at least one index $i$.
Hence again $\hat\nabla_{\Lambda_{0\dots0}}$ is divisible by $z^m$.

Since $\hat\nabla_{\Lambda_{0\dots0}}$ is divisible by $z^m$, $\hat\nabla^\x_\Lambda$ is divisible by $z^{m+2n}$.
Since $\Lambda$ was an arbitrary singular link with $2n$ double points whose all double points are self-intersections 
of components, the coefficient of $\hat\nabla_L$ at $z^{m-1+2n}$ is a colored type $2n-1$ invariant.
\end{proof}

\begin{proposition} \label{ctype2} $\alpha_2(L)$ is not of colored type $2$, nor even of type $(2,2)$ for two-component links 
of linking number $0$.
\end{proposition}

\begin{proof} For a two-component link $L$ of linking number $0$ we have $c_0(L)=0$ and hence 
$\alpha_2(L)=c_2(L)-c_1(L)\big(c_1(K_1)+c_1(K_2)\big)$.
Let $L_{\x\x\x}$ be a two-component singular link with $3$ double points and linking number $0$ such that all its double points 
are on the first component, and after smoothing any two of them we get a $2$-component (and not $4$-component) link.
If we regard $v_i(L):=c_1(K_i)$ as an invariant of $L$, then we have 
$\alpha_2^\x(L_{\x\x\x})=c_2^\x(L_{\x\x\x})-(c_1v_1)^\x(L_{\x\x\x})-(c_2v_2)^\x(L_{\x\x\x})$.
Clearly $c_2^\x(L_{\x\x\x})=c_0(L_{000})$, where $L_{000}$ is the $3$-component link obtained by smoothing all double points 
of $L_{\x\x\x}$.
Lemma \ref{leibniz} expresses $(c_1v_i)^\x(L_{\x\x\x})$ as the sum of all products 
of the form $c_1^\x(L_\alpha)v_i^\x(L_\beta)$, where $L_\alpha$ is obtained by resolving positively some number $i$ of 
the double points of $L_{\x\x\x}$ and $L_\beta$ by resolving negatively the remaining $3-i$ double points of $L_{\x\x\x}$.
By Lemma \ref{conway-coefficients} $c_1(L)$ is of colored type $2$, so $c_1^\x(L_{\x\x\x})=0$.
But since $v_2(L)$ depends only on the second component, $v_2^\x$ vanishes on any singular link that has at least one 
double point on the first component, so from Lemma \ref{leibniz} we get $(c_1v_2)^\x(L_{\x\x\x})=0$.
Since $v_1$ is of colored type $2$ (cf.\ Lemma \ref{knot-link}(a)), we have $v_1^\x(L_{\x\x\x})=0$.
If the first double point of $L_{\x\x\x}$ is resolved positively or negatively, 
then the resulting singular link $L_{\pm\x\x}$ satisfies $c_1^\x(L_{\pm\x\x})=c_0(L_{\pm00})$, where $L_{\pm00}$ is the 
two-component link obtained by smoothing both double points of $L_{\pm\x\x}$, but $c_0(L_{\pm00})=\lk(L_{\pm00})=\lk(L_{\pm\x\x})=0$.
Also, $v_1^\x(L_{\pm\x\x})=c_0(K_{\pm00})$, where $K_{\pm00}$ is the first component of $L_{\pm00}$,
but we know that $c_0(K_{\pm00})=1$ (see Lemma \ref{lickorish}(c)).
Similar arguments apply if the second or the third double point of $L_{\x\x\x}$ is resolved positively or negatively.
Thus from Lemma \ref{leibniz} we get 
\begin{multline*}
(c_1v_1)^\x(L_{\x\x\x})=c_1^\x(L_{++\x})v_1^\x(L_{\x\x-})+c_1^\x(L_{+\x+})v_1^\x(L_{\x-\x})+c_1^\x(L_{\x++})v_1^\x(L_{-\x\x})\\
=c_1^\x(L_{++\x})+c_1^\x(L_{+\x+})+c_1^\x(L_{\x++})=c_0(L_{++0})+c_0(L_{+0+})+c_0(L_{0++}),
\end{multline*}
where $+$ and $-$ indicate the positive and the negative resolution of the corresponding double point, and $0$ its smoothing.
Thus \[\alpha_2^\x(L_{\x\x\x})=c_0(L_{000})-c_0(L_{++0})-c_0(L_{+0+})-c_0(L_{0++}).\]

Let $\phi\:S^1\to\R^2$ be a self-transverse $C^1$-approximation with $3$ double points of the double cover 
$S^1\to S^1$ of the clockwise oriented unit circle $S^1\subset\R^2$, and let $A,B,C,D$ denote the $4$ bounded components of 
$\R^2\but\phi(S^1)$ such that $D$ contains the origin.
The composition $K_{\x\x\x}\:S^1\xr{\phi}\R^2\subset\R^3$ is a singular knot.
Let $K'\:S^1\emb S^3\but K_{\x\x\x}(S^1)$ be a knot in the complement of $K_{\x\x\x}$ linking the clockwise oriented 
boundaries of $A,B,C,D$ with linking numbers $a,b,c,d$, respectively, so that $a+b+c+2d=0$.
Finally, define $L_{\x\x\x}$ to be the union of $K_{\x\x\x}$ and $K'$.
Then $\lk(L_{\x\x\x})=a+b+c+2d=0$ and by smoothing any two of the $3$ double points of $L_{\x\x\x}$ 
we get a $2$-component link, so that our previous formula for $\alpha_2^\x(L_{\x\x\x})$ applies.
It is easy to see that by smoothing any one of the double points of $K_{\x\x\x}$, we get a link of linking number $1$,
and by smoothing all three we get a link of linking number $0$.
Then using the formula for $c_0$ of a $3$-component link (see Remark \ref{hhh}), we find that
$c_0(L_{000})=(a+b+c+d)d$ and one of $c_0(L_{++0})$, $c_0(L_{+0+})$ and $c_0(L_{0++})$ equals 
$(a+b+d)(c+d)+(a+b+d)+(c+d)$ and the other two are obtained from it by cyclically permuting $a$, $b$ and $c$.
Since $(a+b+d)+(c+d)=0$ and $(a+b+d)(c+d)=-(c+d)^2$, we get
\[\alpha_2^\x(L_{\x\x\x})=(a+b+c+d)d+(a+d)^2+(b+d)^2+(c+d)^2.\]
This expression is nonzero e.g.\ for $a=1$, $b=2$, $c=-3$, $d=0$. 
\end{proof}

It was shown in \cite{MR1}*{proof of Theorem 2.2} that $\alpha_1$ is invariant under $1$-quasi-isotopy.
By Proposition \ref{reduced-coefficients2} and Theorem \ref{2.2} $\alpha_n(L)$ is invariant under
$(2n-1)$-quasi-isotopy.
However, by Proposition \ref{ctype2} the following cannot be deduced from Theorem \ref{2.2}.

\begin{theorem} \label{conway-quasi} $\alpha_n(L)$ is invariant under $n$-quasi-isotopy.
\end{theorem}

\begin{proof} Let $L$ be a singular link which is an $n$-quasi-embedding, and let $m$ be the number of its components.
Let $P_0,\dots,P_n$ and $J_0,\dots,J_n$ be as in the definition of an $n$-quasi-embedding.
We may assume that the $P_i$ are connected.
By passing to small regular neighborhoods we may assume that $P_1,\dots,P_n$ are connected $3$-manifolds with boundary,
and $P_0$ lies in a $3$-ball $P_0^+$ such that $P_0^+\cup L(J_0)$ is null-homotopic in $P_1$.
We may assume that the images of $L_+$, $L_-$ and $L_0$ lie in $P_0^+\cup L$.
For the purposes of the present proof let us redefine $P_0$ as $P_0^+$.
Let $K$ be the component of $L$ containing the double point.
Then $K_+$, $K_-$ and $K_0$ lie in $P_0\cup K=P_0\cup K_-$.

\begin{lemma} \cite{M24-3}*{Corollary \ref{part3:annulus}} 
\label{annulus0}
Let $L$ be a link in $S^3$ and let $A$ be an embedded annulus in $S^3\but L$ whose core has zero linking number with $L$.
Then $\nabla_{L\cup\partial A}(z)=0$.
\end{lemma}

\begin{lemma} \label{trees}
For each $k\le n$ there exists a computation tree $T_{2k}$ of order $2k$ for $K_+$, whose branches are homotopies with values 
in $P_k\cup K_-$ and whose leaves are split links, copies of $K_-$ and $3$-component links of the form $K_-\cup\partial A$, where $A$ is 
an embedded annulus in $P_k\but K_-$ whose core has zero linking number with $K_-$. 
\end{lemma}

\begin{proof} Let $T_0$ consist of the straight line homotopy between $K_+$ and $K_-$, which has one singular link, namely, $K$.
Arguing by induction, we may assume that $T_{2k}$ has been constructed.
Let $X$ be a bud of $T_{2k}$.
Let $Q_1,\dots,Q_\mu$ be the components of its smoothing $X_0$.
Thus $Q_1,\dots,Q_\mu$ lie in $P_k\cup K_-$, and one of them, say $Q_1$, contains $K_-\but P_k$.
Then $Q_1$ is homotopic to $K_-$ in $P_{k+1}\cup K_-$ by some homotopy $Q_{1,t}$, and the other ones are null-homotopic 
in $P_{k+1}$ by some homotopies $Q_{2,t},\dots,Q_{\mu,t}$.
Since $K_+$ is a knot, $X$ has an odd number of components, and then $X_0$ has an even number of components.
Thus $\mu\ge 2$.
Using the null-homotopy $Q_{2,t}$ and that $P_{k+1}$ is a $3$-manifold, it is easy to construct a homotopy $h_t^X$ with values in 
$P_{k+1}\cup K_-$ from $X_0$ to a split link.
Using again that $P_{k+1}$ is a $3$-manifold, we may assume that this homotopy goes only through links and singular links 
with one double point.
Let $T_{2k+1}$ consist of $T_{2k}$ and the homotopies $h_t^X$ for all buds $X$ of $T_{2k}$.
Let $Y$ be a bud of $T_{2k+1}$, that is, a singular link of some $h_t^X$.

Let us first consider the case where the double point of $Y$ is an intersection of $Q_{2,s}$ for some $s\in(0,1)$ with 
some $Q_i$, $i\ne 2$.
If $i\ge 3$, then the null-homotopies $Q_{2,t}|_{t\in [s,1]}$, and $Q_{i,t}$ can be combined to get a null-homotopy 
of $Q_{2,s}\cup Q_i$ in $P_{k+1}$.
This in turn yields a homotopy with values in $P_{k+1}\cup K_-$ from $Y_0$ to a split link.
If $i=1$ and $\mu=2$, then using the same homotopies $Q_{2,t}|_{t\in [s,1]}$ and $Q_{i,t}$ one can similarly construct
a homotopy with values in $P_{k+1}\cup K_-$ from $Y_0$ to $K_-$.
If $i=1$ and $\mu\ge 3$, then using the null-homotopy $Q_{3,t}$ it is easy to construct a homotopy with values in $P_{k+1}\cup K_-$ 
from $Y_0$ to a split link.

It remains to consider the case where the double point of $Y$ is a self-intersection of $Q_{2,s}$ for some $s\in(0,1)$.
If $\mu\ge 3$, then using the null-homotopy $Q_{3,t}$ it is easy to construct a homotopy with values in $P_{k+1}\cup K_-$ 
from $Y_0$ to a split link.
If $\mu=2$, then $Y_0$ is a $3$-component link consisting of $Q_0$ and two other knots $C$, $D$, and
the null-homotopy $Q_{2,t}|_{t\in[s,1]}$ can be used construct a homotopy with values in $P_{k+1}\cup K_-$ from $C$ to $-D$,
i.e.\ to $D$ with orientation reversed.
This in turn yields a homotopy with values in $P_{k+1}\cup K_-$ from $Y_0$ to a link of the form $Q_1\cup\partial A$, where
$A$ is an embedded annulus in $P_{k+1}$, disjoint from $Q_1$.
Since $P_{k+1}$ is connected, by a further homotopy with values in $P_{k+1}\cup K_-$ we may assume that the core of $A$ has 
zero linking number with $K_-$. 
\end{proof}

Let $T_{2n}$ be a computation tree given by Lemma \ref{trees}.
Since $K_0$ is a $2$-component link, each leaf of $T_{2n}$ of an odd order must have an even number of components, so it must be a split link.
Let $Z$ be a leaf of $T_{2n}$ of order $2k$ which is of the form $K_-\cup\partial A$.
Then the annulus $A$ lies in $P_k$, so in the case $k<n$ is it null-homotopic in $P_n$.
Then its core $\alpha$ has zero linking number with $L_-\but K_-$.
Since $\alpha$ also has zero linking number with $K_-$, we have $\lk(\alpha,\,L_-)=0$.
This does not work if $Z$ is a leaf of order $2n$, but in this case, since $P_n$ is connected, we can amend $Z$ by a homotopy 
with values in $P_n\cup K_-$ so that $\lk(\alpha,\,K_-)$ is any given number.
If we choose this number to be $-\lk(\alpha,\,L_-\but K_-)$, then we get $\lk(\alpha,\,L_-)=0$.
This amendment, when carried out for all leaves of $T_{2n}$ of order $2n$ which are of the form $K_-\cup\partial A$, results in a new
computation tree $T_{2n}'$, whose branches of orders $<2n$ are the same as those of $T_{2n}$.
By adding the identical homotopies of $L_-\but K_-$, we may regard $T_{2n}'$ as a computation tree of order $2n$ for $L_+$, whose leaves 
are split links, copies of $L_-$, and links of the form $L_-\cup\partial A$, where $A$ is an embedded annulus in $S^3\but L_-$ whose core
has zero linking number with $L_-$. 

Taking into account Lemma \ref{annulus0}, from $T_{2n}$ we get 
\[\nabla_{K_+}=\sum_{k=0}^n z^{2k}\Big(\sum_{i=1}^{r_k}\epsilon_{ki}\nabla_{K_-}\Big)+z^{2n+1}P\] for some $P\in\Z[z]$, 
where $r_0=1$ and $\epsilon_{01}=+1$, and for $k>0$ the signs $\epsilon_{ki}$ are the signs of some of the singular links in the homotopies 
that are branches of $T_{2n}$ of order $2k-1$.
The polynomial $P$ is a sum with signs of the Conway polynomials of the buds of $T_{2n}$.
Similarly, from $T_{2n}'$ we get \[\nabla_{L_+}=\sum_{k=0}^n z^{2k}\Big(\sum_{i=1}^{r_k}\epsilon_{ki}\nabla_{L_-}\Big)+z^{2n+1}Q\] 
for some $Q\in\Z[z]$, 
where the $r_k$ and the $\epsilon_{ki}$ are the same as in the previous sum, since they are determined by the branches of orders $<2n$, which are
shared by $T_{2n}$ and $T_{2n}'$.
The polynomial $Q$ is a sum with signs of the Conway polynomials of the buds of $T_{2n}'$.
These buds all have at least $m$ components, so by Proposition \ref{lickorish}(b) $Q=z^{m-1}Q_0$ for some $Q_0\in\Z[z]$.
Let $R=\sum_{k=0}^n z^{2k}\sum_{i=1}^{r_k}\epsilon_{ki}$.
Then $\nabla_{K_+}=R\nabla_{K_-}+z^{2n+1}P$ and $\nabla_{L_+}=R\nabla_{L_-}+z^{(m-1)+(2n+1)}Q_0$.
Hence
\[\frac{z^{1-m}\nabla_{L_+}}{\nabla_{K_+}}=
\frac{z^{1-m}\nabla_{L_-}R+z^{2n+1}Q_0}{\nabla_{K_-}R+z^{2n+1}P}=
\frac{z^{1-m}\nabla_{L_-}}{\nabla_{K_-}}+z^{2n+1}S\]
for some $S\in\Z[[z]]$.
Therefore $z^{1-m}\bar\nabla_{L_+}=z^{1-m}\bar\nabla_{L_-}+z^{2n+1}S_0$ for some $S_0\in\Z[[z]]$, and the assertion follows.
\end{proof}

\begin{remark} \label{conway-quasi-discussion} 
The most difficult case in the proof of Lemma \ref{trees} is a self-intersection of the null-homotopy of $Q_2$
in the absence of $Q_3$.
However, if we start from a {\it strong} $k$-quasi-isotopy, then this difficulty disappears (because a knot lying in a ball 
is null-homotopic within this ball), and one might even wonder if anything at all prevents building a computation tree of 
any order under the hypothesis of strong $1$-quasi-isotopy.
In fact, the only thing which may go wrong with it is the following: in the homotopy from $Q_1$ to $K_-$ we may get 
a self-intersection with $Q_1$.
If one lobe of the resulting singular link lies entirely within the ball $B_1$ (from the definition of a strong $1$-quasi-isotopy),
this is still not a problem, but it may happen that none of the two lobes lies in $B_1$.
\end{remark}

Since $c_k\equiv\alpha_k\pmod{\gcd(c_0,\dots,c_{k-1})}$, Theorem \ref{conway-quasi} implies the case $l=1$, and also 
the $2$-component case of the following

\begin{theorem} \label{conway-mu} \cite{MR2}*{Corollary 3.5}
Set $\lambda=\lceil\frac{(l-1)(m-1)}2\rceil$.
The residue class of $c_{\lambda+k}$ modulo $\gcd$ of $c_\lambda,\dots,c_{\lambda+k-1}$ 
and all $\bar\mu$-invariants of length $\le l$ is invariant under
$(\lfloor\frac\lambda{m-1}\rfloor+k)$-quasi-isotopy of $m$-component links.

(Here $\lceil x\rceil=n$ if $x\in [n-\frac12,n+\frac12)$, and
$\lfloor x\rfloor=n$ if $x\in (n-\frac12,n+\frac12]$ for $n\in\Z$.)
\end{theorem}

One special case of Theorem \ref{conway-mu} not covered by Theorem \ref{conway-quasi} asserts that for $3$-component
links the residue class of $c_k$ modulo the greatest common divisor $\Delta_k$ of all $\bar\mu$-invariants of length 
$\le k+1$ is invariant under $\lfloor\frac k2\rfloor$-quasi-isotopy.
When $k=1$, this is saying that the residue class of $c_1$ modulo the $\gcd$ of the pairwise linking numbers
is invariant under $0$-quasi-isotopy, that is, link homotopy.
This residue class is also known as $\bar\mu(123)^2$ (Murasugi--Traldi \cite{Ms}, \cite{Tr1} and Cochran \cite{Co2};
see also \cite{Le}, \cite{Masb}).

Can one find an integer lift of $\bar\mu(123)^2$, invariant under link homotopy, among coefficients of rational functions
in $\bar\nabla_\Lambda$ for sublinks $\Lambda$ of $L$?
The answer is negative, since the desired integer invariant would be a finite type (specifically, type $4$) invariant 
of link homotopy which is not a function of the pairwise linking numbers, but this is impossible for $3$-component 
links by \cite{MT}.
Alternatively, using Proposition \ref{gamma} one can see it directly, by checking that the jump of $\gamma(L)$ 
cannot be canceled by the jump of any polynomial expression in the coefficients of $\nabla_\Lambda$, 
homogeneous of degree $4$.

\section*{Acknowledgements}

I would like to thank P. M. Akhmetiev, I. Dynnikov, M. Il'insky and L. Plachta for useful discussions and thoughtful remarks.



\end{document}